\documentclass{amsart}

\usepackage{amsmath,amsfonts,amssymb,latexsym,mathrsfs,shuffle}
\usepackage{tikz}
\usepackage{cancel}
\usetikzlibrary{decorations.markings}

\usepackage{color}
\newcommand{\Blue}[1]{\textcolor{blue}{#1}}
\newcommand{\Red}[1]{\textcolor{red}{#1}}

\newtheorem{theorem}{Theorem}[section]
\newtheorem{lemma}[theorem]{Lemma}
\newtheorem{proposition}[theorem]{Proposition}
\newtheorem{corollary}[theorem]{Corollary}

\theoremstyle{definition}
\newtheorem{definition}[theorem]{Definition}
\newtheorem{example}[theorem]{Example}

\theoremstyle{remark}
\newtheorem{remark}[theorem]{Remark}

\numberwithin{equation}{section}

\renewcommand{\P}{\mathcal{P}}
\renewcommand{\L}{\mathcal{L}}

\newcommand{\precdot}{\mathrel{\ooalign{$\prec$\cr\hidewidth\hbox{$\cdot$}\cr}}}

\newcommand{\A}{\mathscr{A}}

\newcommand{\Des}{\mathrm{Des}}   
\newcommand{\des}{\mathrm{des}}
\newcommand{\red}{{\mathrm{rdes}}}
\newcommand{\rhomax}[1]{{\overline{\rho}^{#1}}}   

\newcommand{\flatten}{\mathrm{flat}}
\newcommand{\comp}{\mathrm{comp}}

\newcommand{\fund}{\mathfrak{F}}

\usepackage{xcolor}
\usepackage[colorinlistoftodos]{todonotes}

\begin{document}


\title[Flagged $(\P,\rho)$-partitions]{Flagged $(\P,\rho)$-partitions}

\author[S. Assaf]{Sami Assaf}
\address{Department of Mathematics, University of Southern California, 3620 S. Vermont Ave, Los Angeles, CA 90089-2532, U.S.A.}
\email{shassaf@usc.edu}
\thanks{S.A. supported in part by NSF DMS-1763336.}

\author[N. Bergeron]{Nantel Bergeron}
\address{Department of Mathematics and Statistics, York University, 4700 Keele St, Toronto, Ontario M3J 1P3, Canada}
\email{bergeron@yorku.ca}
\thanks{N.B. supported in part by York Research Chair in Applied Algebra and NSERC.}





\begin{abstract}
  We introduce the theory of $(\P,\rho)$-partitions, depending on a poset $\P$ and a map $\rho$ from $\P$ to positive integers. The generating function $\fund_{\P,\rho}$ of $(\P,\rho)$-partitions is a polynomial that, when the images of $\rho$ tend to infinity, tends to Stanley's generating function of $\P$-partitions. Analogous to Stanley's fundamental theorem for $\P$-partitions, we show that the set of $(\P,\rho)$-partitions decomposes as a disjoint union of $(\L,\rho)$-partitions where $\L$ runs over the set of linear extensions of $\P$. In this more general context, the set of all $\fund_{\L,\rho}$ for linear orders $\L$ over determines a basis of polynomials. We thus introduce the notion of flagged $(\P,\rho)$-partitions, and we prove that the set of all $\fund_{\L,\rho}$ for flagged $(\L,\rho)$-partitions for linear orders $\L$ is precisely the fundamental slide basis of the polynomial ring, introduced by the first author and Searles. Our main theorem shows that any generating function $\fund_{\P,\rho}$ of flagged $(\P,\rho)$-partitions is a positive integer linear combination of slide polynomials. As applications, we give a new proof of positivity of the slide product and, motivating our nomenclature, we also prove flagged Schur functions are slide positive.
\end{abstract}

\maketitle
\tableofcontents

%
\section{Introduction}
%
\label{sec:intro}

The theory of $\P$-partitions has its origins with MacMahon \cite{Mac60} and was developed in depth by Stanley \cite{Sta72}. There have been several applications of this theory. Notably, the quasisymmetric functions developed by Gessel \cite{Ges84} originated from the study of $\P$-partitions. Symmetric functions are a special case of quasisymmetric functions as can be seen by realizing Schur functions as certain $\P$-partition generating functions. 

More recently, Assaf and Searles~\cite{AS17} introduced slide polynomials to study Schubert polynomials \cite{LS82}. The slide polynomials have several interesting properties related to Schubert polynomials and certain limits related to quasisymmetric functions. Moreover, Assaf and Searles show that the product of slide polynomials expands positively in term of slide polynomials. We see in our present work that slide polynomials arise naturally in a restricted version of $\P$-partitions, and, moreover, the positivity of the slide product naturally follows.

A \emph{poset} $\P$ of order $p$ is a partial order on the set $[p] = \{1,2,\ldots,p\}$.  We have several orders that come into play: $\prec$ denotes the partial order on $\P$, $\precdot$ denotes a cover relation in $\P$, and $<$ denotes the natural order of the integers. Given $\rho\colon [p]\to {\mathbb Z}$ any map, a function $f\colon \P\to{\mathbb N}$ is a \emph{$(\P,\rho)$-partition} if 
\begin{enumerate}
\item if $i \prec j$, then $f(i) \le f(j)$, 
\item if $i \prec j$ and $i > j$, then $f(i) < f(j)$, and 
\item $f(i)\le \rho(i)$.
\end{enumerate}
Here the first two conditions alone characterize classical $\P$-partitions \cite{Sta72}.
  
The generating function $\fund_{\P,\rho}$ of $(\P,\rho)$-partitions is the polynomial
$$ \fund_{\P,\rho} = \sum_{f\in \A_\rho(\P)}  x_{f(1)} \cdots x_{f(p)} $$
where $\A_\rho(\P)$ denotes the set of all $(\P,\rho)$-partitions. In Section~\ref{sec:FTRPP} we show that we have a decomposition
\begin{equation}
  \A_{\rho}(\P) = \bigsqcup_{\L \in \mathscr{L}(\P)} \A_{\rho}(\L),
\end{equation}
where the disjoint union is over the set $\mathscr{L}(\P)$ of linear extensions of $\P$. This implies that $\fund_{\P,\rho}$ is a positive linear combination of $ \fund_{\L,\rho} $ where $\L$ are linear orders. Unfortunately, the set $\{ \fund_{\L,\rho} \}$ as $\L$ runs over all linear orders on $[p]$ and $\rho$ is any restriction does not form a basis of polynomials. This leads us to restrict further the restriction maps $\rho$ that we use.
  
We say that a restriction map $\rho$ is a \emph{$\P$-flag} if $\rho(i)>0$ for all $i\in[p]$, and
  \begin{enumerate}
  \item[(i)] if $i \precdot j$ and $i<j$, then $\rho(i) = \rho(j)$, and
  \item[(ii)] if $i \precdot j$ and $i > j$, then $\rho(i) \le \rho(j)$.
    \end{enumerate}
  When $\rho$ is a \emph{$\P$-flag}, we say that $\A_{\rho}(\P)$ is the set of \emph{flagged $(\P,\rho)$-partitions}. In Section~\ref{sec:linear}, we show that the set $\{ \fund_{\L,\rho} \}$ as $\L$ runs over all linear orders on $[p]$ and $\rho$  over all $\L$-flags is exactly the set of fundamental slide polynomials as defined in~\cite{AS17}. In particular this set is a basis for polynomials. The main theorem of this paper, proved in Section~\ref{sec:mainproof}, states:
  
\begin{theorem} \label{thm:main}
For any $\P$ on $[p]$ and $\P$-flag $\rho$, the polynomials $\fund_{\P,\rho}$ expand positively in the basis of slide polynomials.
\end{theorem}

We note that if $\rho$ is not a $\P$-flag, then theorem above is false. Even for linear orders $\L$, if $\rho$ is not a $\L$-flag, then $\fund_{\L,\rho}$ will in general have negative coefficients. 

There are several consequences of our main theorem; we mention two. First, an immediate application of the theorem is a new proof that the product of two slide polynomials is a positive linear combination of slide polynomials. Second, the theorem shows that the flagged Schur functions of Lascoux and Sch\"{u}tzenberger~\cite{LS82,wachs82} expand positively in term of slides polynomials. The latter motivates for the term \emph{flag} used for the subset of $(\P,\rho)$ partitions that we study. 

%
\section{Stanley's theory of $\P$-partitions}
%
\label{sec:stanley}

Before delving into the details of this new theory of $(\P,\rho)$-partitions, we begin with a brief review of Stanley's theory of $\P$-partitions and its consequences. For a comprehensive survey of $\P$-partitions, see \cite{Ges16}. 

\begin{definition}[\cite{Sta72}]\label{def:P-partition}
  Given a poset $\P$ on $[p]$, a \emph{$\P$-partition} is a map $f : \P \rightarrow \mathbb{N}$ such that for all $i,j\in\P$ we have
  \begin{enumerate}
  \item if $i \prec j$, then $f(i) \le f(j)$, and
  \item if $i \prec j$ and $i > j$, then $f(i) < f(j)$.
  \end{enumerate}
  If $\sum_{i \in \P} f(i) = n$, then $f$ is a \emph{$\P$-partition of $n$}.
\end{definition}


For example, the poset with Hasse diagram shown in Fig.~\ref{fig:poset} has edges indicating relations between $f(i)$ and $f(j)$ that any $\P$-partition $f$ must satisfy.

\begin{figure}[ht]
  \begin{center}
    \begin{tikzpicture}[scale=1,
        roundnode/.style={circle, draw=black, thick, minimum size=1ex},
        label/.style={%
          postaction={ decorate,transform shape,
            decoration={ markings, mark=at position .5 with \node #1;}}}]
      \node[roundnode] at (0,3)   (P5) {$5$};
      \node[roundnode] at (1,1.5) (P1) {$1$};
      \node[roundnode] at (2,0)   (P7) {$7$};
      \node[roundnode] at (2,3)   (P2) {$2$};
      \node[roundnode] at (3,1)   (P4) {$4$};
      \node[roundnode] at (3,2)   (P6) {$6$};
      \node[roundnode] at (4,0)   (P3) {$3$};
      \draw[thin,label={[below]{\Blue{$\scriptstyle \ge$}}}] (P5) -- (P1) ;
      \draw[thin,label={[above]{\Blue{$\scriptstyle \le$}}}] (P1) -- (P2) ;
      \draw[thin,label={[above]{\Blue{$\scriptstyle \ge$}}}] (P4) -- (P3) ;
      \draw[thin,label={[above]{\Blue{$\scriptstyle \ge$}}}] (P6) -- (P4) ;
      \draw[thin,label={[above]{\Blue{$\scriptstyle >$}}}] (P2) -- (P6) ;
      \draw[thin,label={[below]{\Blue{$\scriptstyle >$}}}] (P1) -- (P7) ;
      \draw[thin,label={[above]{\Blue{$\scriptstyle >$}}}] (P4) -- (P7) ;
    \end{tikzpicture}
  \end{center}
  \caption{\label{fig:poset}An example of a labeled poset of order $7$, with edges decorated by the conditions on a map to be a $\P$-partition.}
\end{figure}
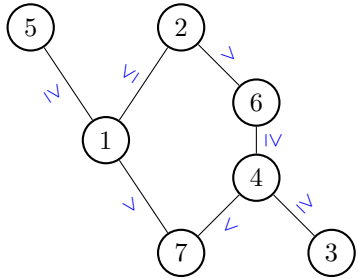

Let $\A(\P)$ denote the set of $\P$-partitions. Then we may define the generating function of a poset $\P$ by
\begin{equation}\label{e:poset-gf}
  F_{\P} = \sum_{f \in \A(\P)} x_{f(1)} \cdots x_{f(p)} .
\end{equation}

\subsection{Linear $\P$-partitions}
\label{sec:stanley-linear}

As we will see below in the Fundamental Theorem of $\P$-partitions, linear orders are of particular interest. 

Given a linear order $\L$, we say that $i \precdot j$ is a \emph{descent} whenever $i>j$. A strong composition is a finite sequence of positive integers, and we use letters $\alpha, \beta, \gamma$ for strong compositions. We can record descents of $\L$ with the descent composition defined as follows.

\begin{definition}\label{def:strong-descent}
  For $\L$ a linear order on $[p]$, the \emph{descent composition of $\L$}, denoted by $\Des(\L)$, is formed by removing the edge in the Hasse diagram between $i$ and $j$ whenever a cover $i \precdot j$ is such that $i>j$, calling the resulting chains $C_1,\ldots,C_r$ taken in ascending order in $\L$, and setting $\Des(\L)_{s} = |C_s|$. We have that $\Des(\L)$ is a strong composition of $p$.
\end{definition}

\begin{example}\label{ex:strong-descent}
  Let $\L$ be the linear order $2 \prec 5 \prec 1 \prec 4 \prec 6 \prec 7 \prec 3 \prec 8 \prec 9$. Then the chains after removing edges for descents become
  \[ \overbrace{2 \prec 5}^{C_1} \hspace{2em} \overbrace{1 \prec 4 \prec 6 \prec 7}^{C_2} \hspace{2em} \overbrace{3 \prec 8 \prec 9}^{C_3} \ . \]
  Therefore $\Des(\L) = (|C_1|,|C_2|,|C_3|) = (2,4,3)$.
\end{example}

Gessel \cite{Ges84} observed that for $\L$ a linear order, the function $F_{\L}$ depends only on $\Des(\L)$. This led him to introduce the \emph{fundamental quasisymmetric functions} \cite{Ges84}, indexed by strong compositions, that form an important basis for quasisymmetric functions. 

Given strong compositions $\alpha,\beta$, we say that \emph{$\beta$ refines $\alpha$} if for all $j$ there exist indices $i_1<\ldots<i_k$ such that
\begin{displaymath}
  \beta_1 + \cdots + \beta_{i_j} = \alpha_1 + \cdots + \alpha_j.
\end{displaymath}
For example, $(1,2,2)$ refines $(3,2)$ since $1+2 = 3$ and $1+2+2 = 3+2$. However, $(1,2,2)$ does not refine $(2,3)$ since $1 < 2$ and $1 + 2 > 2$.

A weak composition is a finite sequence of non-negative integers, and we use letters $a, b, c$ for weak compositions. The flattening of a weak composition is the strong composition obtained by removing the zeros. For example, $\flatten(0,3,0,2) = (3,2)$.
We will also use  the \emph{dominance order} on weak compositions defined by $b \trianglerighteq a$ if and only if $b_1 + \cdots + b_k \geq a_1 + \cdots + a_k$ for all $k$. Note that this is a partial order, and the linear order given by reverse lexicographic order extends dominance order.

\begin{definition}[\cite{Ges84}]
  For $\alpha$ a strong composition, the \emph{fundamental quasisymmetric function} $F_{\alpha}$ is given by
  \begin{equation}
    F_{\alpha}(X) = \sum_{\flatten(b) \ \mathrm{refines} \ \alpha} x_1^{b_1} x_2^{b_2} \cdots ,
    \label{e:F_n}
  \end{equation}
  where the sum is over weak compositions $b$ whose flattening refines $\alpha$.
\end{definition}
  
For example, restricting to three variables to make the expansion finite, we have
\begin{displaymath}
  F_{(3,2)}(x_1,x_2,x_3) = x_2^3 x_3^2 + x_1^3 x_3^2 + x_1^3 x_2^2 + x_1^3 x_2 x_3 + x_1 x_2^2 x_3^2 + x_1^2 x_2 x_3^2.
\end{displaymath}

\begin{proposition}[\cite{Ges84}]\label{prop:gessel}
  For $\L$ a linear order on $[p]$, we have
  \begin{equation}
    F_{\L} = F_{\Des(\L)},
  \end{equation}
  where $F_{\alpha}$ denotes the fundamental basis for quasisymmetric functions.
\end{proposition}

\subsection{Fundamental Theorem for $\P$-partitions}
\label{sec:stanley-fundamental}

Given a poset $\P$, a \emph{linear extension of $\P$} is a linear order on $[p]$ that extends $\P$. For example, the linear order $7 \prec 1 \prec 3 \prec 5 \prec 4 \prec 6 \prec 2$ is a linear extension of the poset in Fig.~\ref{fig:poset}.

Stanley proved the following result that has myriad consequences for $\P$-partitions and their generating functions.

\begin{theorem}[Fundamental Theorem of $\P$-partitions \cite{Sta72}]\label{thm:fundamental}
  Given a poset $\P$ on $[p]$, we have
  \begin{equation}\label{e:A-fund}
    \A(\P) = \bigsqcup_{\L \in \mathscr{L}(\P)} \A(\L),
  \end{equation}
  where the disjoint union is over the set $\mathscr{L}(\P)$ of linear extensions of $\P$. 
\end{theorem}

In particular, this gives a simple decomposition of the generating function of a poset as the sum of generating functions of all linear extensions of the poset, the latter of which are elements of Gessel's fundamental basis.

\begin{corollary}\label{cor:fundamental}
  Given a poset $\P$ on $[p]$, we have
  \begin{equation}\label{e:F-fund}
    F_{\P} = \sum_{\L \in \mathscr{L}(\P)} F_{\L} = \sum_{\L \in \mathscr{L}(\P)} F_{\Des(\L)}.
  \end{equation}  
\end{corollary}

For example, there are $18$ linear extensions of the poset in Fig.~\ref{fig:poset}, and so the corresponding generating function is a sum of $18$ terms of the fundamental basis.

An immediate and powerful consequence of Theorem~\ref{thm:fundamental} is a formula for the fundamental expansion of a product of elements of the fundamental basis.

\begin{corollary}\label{cor:shuffle}
  Let $\alpha, \beta$ be two strong compositions. Choose two linear orders $A$ and $B$ such that $\alpha=\Des(A)$ and $\beta=\Des(B)$. Then
  \begin{equation}\label{e:shuffle}
    F_{\alpha} F_{\beta} = F_{A} F_{B} = F_{A \oplus B} =\sum_{\L \in \mathscr{L}(A \oplus B)} F_{\Des(\L)},
  \end{equation}
  where $A \oplus B$ denotes the partial order given by the disjoint union of $A$ and $B$ and where no element of $A$ is comparable to an element of $B$. The set $\mathscr{L}(A \oplus B)$ corresponds to the shuffle product of $A$ and $B$.
\end{corollary}

\begin{example}\label{ex:shuffle}
  Let $\alpha = (2)$ and $\beta = (2)$, and set $A = (1\prec 2)$ and $B = (3\prec 4)$. We do have that $\Des(A)=\alpha$ and $\Des(B)=\beta$.
  Since $A$ and $B$ are linear orders, let us represent them using words. That is $A=12$ and $B=34$.
  Then
  \[ A \shuffle B = \left\{ 1234, \ 1324, \ 1342, \ 3124, \ 3142, \ 3412 \right\}, \]
  and the words we get are exactly the linear extensions of $\mathscr{L}(A \oplus B)$. They are distinct but could have the same descent composition. For example $\Des(1324)=\Des(3412)=(2,2)$, hence the coefficient of 
  $F_{(2,2)}$ in the product is 2.
We have
  \[ F_{(2)} F_{(2)} = F_{A \oplus B} = F_{(4)} + 2 F_{(2,2)} + F_{(3,1)} + F_{(1,3)}+ F_{(1,2,1)} \]
\end{example}

\begin{remark}\label{rem:quasi}
At this point, it is important to notice that if we order the monomials lexicographically, then the leading term of $F_\alpha$
is $x_1^{\alpha_1}x_2^{\alpha_2}\cdots x_\ell^{\alpha_\ell}$
for any strong composition $\alpha=(\alpha_1, \alpha_2,\ldots,\alpha_\ell)$. Hence the leading terms for  the $F_\alpha$
are all distinct as $\alpha$ runs over all strong composition. This shows that the set $\big\{ F_\alpha\big\}$ is linearly independent.
Corollary~\ref{cor:shuffle} shows that this set spans all possible products of $F_\alpha$'s. Hence the set $\big\{ F_\alpha\big\}$
is a basis of the algebra it generates.
\end{remark}

For another compelling example, consider the labeled poset $\P_{\lambda}$ associated to a partition $\lambda$, as illustrated in Fig.~\ref{fig:partition}. The generating function for $\P_{\lambda}$ precisely enumerates semistandard Young tableaux, and so
\begin{equation}\label{e:schur}
  F_{\P_{\lambda}} = s_{\lambda},
\end{equation}
is the Schur function corresponding to $\lambda$. Linear extensions of $\P_{\lambda}$ are in bijection with standard Young tableaux, and if $\L$ corresponds to $T$, then $\Des(\L) = \Des(T)$. Therefore the Fundamental Theorem gives the following alternative expansion for Schur functions in terms of standard Young tableaux first shown in \cite{Ges84}.

\begin{figure}[ht]
  \begin{center}
    \begin{tikzpicture}[
        roundnode/.style={circle, draw=black, thick, minimum size=1ex},
        label/.style={%
          postaction={ decorate,transform shape,
            decoration={ markings, mark=at position .5 with \node #1;}}}]
      \node[roundnode] at (3,3) (A5) {$4$};
      \node[roundnode] at (5,3) (A1) {$8$};
      \node[roundnode] at (0,2) (A8) {$1$};
      \node[roundnode] at (2,2) (A6) {$3$};
      \node[roundnode] at (4,2) (A2) {$7$};
      \node[roundnode] at (1,1) (A7) {$2$};
      \node[roundnode] at (3,1) (A3) {$6$};
      \node[roundnode] at (2,0) (A4) {$5$};
      \draw[thin,label={[above]{\Blue{$\scriptstyle \leq$}}}] (A4) -- (A3) ;
      \draw[thin,label={[above]{\Blue{$\scriptstyle \leq$}}}] (A3) -- (A2) ;
      \draw[thin,label={[above]{\Blue{$\scriptstyle \leq$}}}] (A2) -- (A1) ;
      \draw[thin,label={[above]{\Blue{$\scriptstyle \leq$}}}] (A7) -- (A6) ;
      \draw[thin,label={[above]{\Blue{$\scriptstyle \leq$}}}] (A6) -- (A5) ;
      \draw[thin,label={[below]{\Blue{$\scriptstyle <$}}}]    (A4) -- (A7) ;
      \draw[thin,label={[below]{\Blue{$\scriptstyle <$}}}]    (A7) -- (A8) ;
      \draw[thin,label={[below]{\Blue{$\scriptstyle <$}}}]    (A3) -- (A6) ;
      \draw[thin,label={[below]{\Blue{$\scriptstyle <$}}}]    (A2) -- (A5) ;
    \end{tikzpicture}
  \end{center}
  \caption{\label{fig:partition}The labeled poset corresponding to the partition $(4,3,1)$.}
\end{figure}
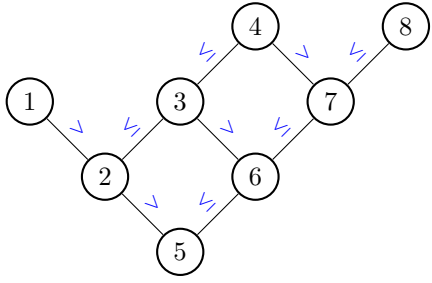

\begin{corollary}\label{cor:syt}
  For $\lambda$ a partition, we have
  \begin{equation}\label{e:syt}
    s_{\lambda} = \sum_{T \in \mathrm{SYT}(\lambda)} F_{\Des(T)} .
  \end{equation}
\end{corollary}

%
\section{ $(\P,\rho)$-partitions}
%
\label{sec:restricted}

We introduce a new generalization of $\P$-partitions by restricting the values of the images of the $\P$-partition with an integer-valued map $\rho$ on $\P$.

\begin{definition}\label{def:restricted-partition}
  Given a poset $\P$ on $[p]$, and a map $\rho : \P \rightarrow \mathbb{Z}$, a \emph{ $(\P,\rho)$-partition} is a $\P$-partition such that $f(i) \leq \rho(i)$ for all $i \in \P$. We denote by $\A_{\rho}(\P)$ the set of all  $(\P,\rho)$-partitions.
\end{definition}

Note that if $\rho(i)\le 0$ for some $i\in[p]$, then $\A_{\rho}(\P)=\varnothing$, and so some restriction maps $\rho$ are too restrictive. Nevertheless, these are still useful to consider.

For a fixed poset $\P$ on $[p]$, some restriction maps $\rho$ are redundant. For example, if we consider the labeled poset in Fig.~\ref{fig:poset-restrict} and a restriction map $\rho$ such that $\rho(3) = 3$ and $\rho(6) = 2$, then since any $\P$-partition $f$ must satisfy $f(3) \leq f(6)$, the restriction on $3$ can never be attained. Such redundancy happens if for two element $i,j\in \P$ we have $i\prec j$ and $\rho(i)>\rho(j)$.

\begin{figure}[ht]
  \begin{center}
    \begin{tikzpicture}[
        roundnode/.style={circle, draw=black, thick, minimum size=1ex},
        label/.style={%
          postaction={ decorate,transform shape,
            decoration={ markings, mark=at position .5 with \node #1;}}}]
      \node[roundnode] at (0,3)   (P5) {$5$};
      \node[roundnode] at (1,1.5) (P1) {$1$};
      \node[roundnode] at (2,0)   (P7) {$7$};
      \node[roundnode] at (2,3)   (P2) {$2$};
      \node[roundnode] at (3,1)   (P4) {$4$};
      \node[roundnode] at (3,2)   (P6) {$6$};
      \node[roundnode] at (4,0)   (P3) {$3$};
      \node[left  = 1pt of P1] {\Blue{$\scriptstyle 4\geq$}};
      \node[right = 1pt of P2] {\Blue{$\scriptstyle\leq 2$}};
      \node[right = 1pt of P3] {\Blue{$\scriptstyle\leq 3$}};
      \node[right = 1pt of P4] {\Blue{$\scriptstyle\leq 2$}};
      \node[left  = 1pt of P5] {\Blue{$\scriptstyle 6\geq$}};
      \node[right = 1pt of P6] {\Blue{$\scriptstyle\leq 2$}};
      \node[left  = 1pt of P7] {\Blue{$\scriptstyle 1\geq$}};
      \draw[thin] (P5) -- (P1) ;
      \draw[thin] (P1) -- (P2) ;
      \draw[thin] (P4) -- (P3) ;
      \draw[thin] (P6) -- (P4) ;
      \draw[thin] (P2) -- (P6) ;
      \draw[thin] (P1) -- (P7) ;
      \draw[thin] (P4) -- (P7) ;
    \end{tikzpicture}
  \end{center}
  \caption{\label{fig:poset-restrict}An example of a poset $\P$ together with a restriction map $\rho$ with restrictions indicated at each node.}
\end{figure}
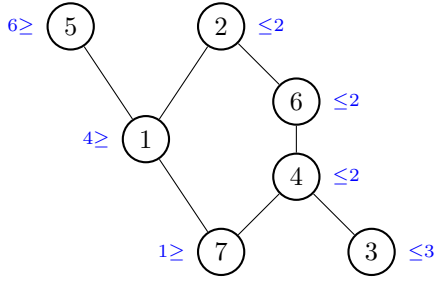

Furthermore, looking at Fig.~\ref{fig:poset-restrict}, we see that if we impose the restriction $\rho(2)=2$, then a $(\P,\rho)$-partition for that $\P$ would need to satisfy
$$1\le f(7) < f(4) \le f(6) < f(2) \le \rho(2)=2$$
which is not possible. The set $\A_{\rho}(\P)$ is empty unless $\rho(i)$ is larger than the size of any decreasing chain in $\P$ ending at $i$. To make this more general, we introduce the following definitions.

For $i \preceq j$ in $\P$, define the \emph{maximum descent distance} $\delta(i,j)$ by
\begin{equation}\label{e:delta}
  \delta(i,j) = \max\{ k \mid  i = i_0\prec i_1\prec\cdots\prec i_k=j \text{ and }  i_0> i_1>\cdots> i_k  \} .
\end{equation}
For example, in Fig.~\ref{fig:poset-restrict} we have $\delta(7,2) = 2$ from the chain $7 \prec 4 \prec 2$.

Using this terminology, given a restriction map $\rho$ for a poset $\P$, we may define a canonical $\P$-increasing map as follows.

\begin{definition}\label{def:rho-max}
  Given a poset $\P$ on $[p]$ and a map $\rho : \P \rightarrow \mathbb{Z}$, the \emph{maximal $\P$-increasing $(\P,\rho)$-partition} $\rhomax{\P}$ is
  \begin{equation}\label{e:max-rho}
    \rhomax{\P}(i) = \min\{ \rho(x) - \delta(i,x) \mid i \preceq x\} .
  \end{equation}
\end{definition}
  
\begin{example}\label{ex:max-rho}
  Let $\P$ be the poset in Fig.~\ref{fig:poset-restrict}, and let $\rho = (4,6,2,3,2,3,3)$ taken in the natural order. Then since $2$ and $5$ are maximal elements, we have $\rhomax{\P}(2) = \rho(2) = 6$ and $\rhomax{\P}(5) = \rho(5) = 2$. The most interesting computation happens for the minimal element $7$, where we have
  \[ \rhomax{\P}(7) = \min\{ \rho(x)-\delta(7,x) \mid x = 7,4,1,6,5,2 \} = 1 \]
  achieved by taking $x=5$. The complete example has $\rhomax{\P}=(2,6,2,3,2,3,1)$, which is the maximal $\P$-increasing $(\P,\rho)$-partition.
\end{example}

Given a poset $\P$ and any map $f: \P \rightarrow \mathbb{N}$, let $\comp(f)$ be the weak composition whose $r$th part is given by $c_r = \#\{ i \in \L \mid f(i) = r \}$. Then for $f \in \A_{\rho}(\P)$, $\comp(f) \trianglerighteq \comp(\rhomax{\P})$ in dominance order. We have the following.

\begin{proposition}\label{prop:max-rho}
  Given any restriction $\rho : \P \rightarrow \mathbb{N}$, we have $\A_{\rho}(\P) \neq \varnothing$ if and only if $\rhomax{\P}(i) \geq 1$ for all $i \in \P$. Furthermore, when this is the case, we have $\A_{\rho}(\P)=\A_{\rhomax{\P}}(\P)$.
\end{proposition}

\begin{proof}
  By definition $\rhomax{\P}(i)$ is the largest possible value that $f(i)$ can take for any $(\P,\rho)$-partition $f$, so if $\rhomax{\P}(i) \leq 0$ for some $i$, then there are no $(\P,\rho)$-partitions. Conversely, if $\rhomax{\P}(i) \geq 1$ for all $i$, then $\rhomax{\P}(i)$ itself is a $(\P,\rho)$-partition. The equality $\A_{\rho}(\P)=\A_{\rhomax{\P}}(\P)$ follow  as $\rhomax{\P}$ is the maximal $(\P,\rho)$-partition for $\rho$, and all other $(\P,\rho)$-partitions must be bounded above by $\rhomax{\P}$.
\end{proof}

We define the generating polynomial of $(\P,\rho)$-partitions by
\begin{equation}\label{e:poset-gp}
  \fund_{(\P,\rho)} = \sum_{f \in \A_{\rho}(\P)} x_{f(1)} \cdots x_{f(p)} .
\end{equation}
By Proposition~\ref{prop:max-rho}, we have $\fund_{(\P,\rho)}=\fund_{(\P,\rhomax{\P})}$ and $\fund_{(\P,\rho)}=0$ unless $\rhomax{\P}(i) \geq 1$ for all minimal elements $i$.

\subsection{Linear  $(\P,\rho)$-partitions}
\label{sec:linear}

The case of linear orders is again of particular interest. To study it, we generalize the notion of descent compositions to account for the restriction map.

\begin{definition}\label{def:reduced-weak-descent}
 For $\L$ a linear order on $[p]$ and $\rho$ a restriction map satisfying $\rhomax{\L}(i) \geq 1$ for all $i \in \L$, we define two associated weak compositions.
 \begin{enumerate}
 \item
 The \emph{reduced weak descent composition of $(\L,\rho)$}, denoted by $\red(\L,\rho)$, is formed as follows. Remove edges between descents, i.e. whenever $i \precdot j$ and $i>j$, and call the resulting chains $C_1,\ldots,C_r$ in ascending order of $\L$. For $s=1,\ldots,r$, set $c_s = \rhomax{\L}(\min_{\L}\{ C_s\})$, and define the part $c_s$ of $\red(\L,\rho)$ to be $\red(\L,\rho)_{c_s} = |C_s|$ and set all other parts to $0$.
  \item
  The \emph{weak descent composition of $(\L,\rho)$}, denoted by $\des(\L,\rho)$, is formed as follows. Remove edges between $i$ and $j$ whenever $i \precdot j$ and $\rhomax{\L}(i) < \rhomax{\L}(j)$, and call the resulting chains $C'_1,\ldots,C'_\ell$ in ascending order of $\L$. For $s=1,\ldots,\ell$, set $c'_s = \rhomax{\L}(\min_\L\{ C'_s\})$, and define the part $c'_s$ of $\des(\L,\rho)$ as $\des(\L,\rho)_{c'_s} = |C'_s|$ and set all other parts to $0$.
  \end{enumerate}
\end{definition}

For Definition~\ref{def:reduced-weak-descent}(1), we have the same chain decomposition as in Definition~\ref{def:strong-descent}. In particular, for any restriction map $\rho$ we have
\begin{equation}\label{e:flat-des}
  \flatten(\red(\L,\rho)) = \Des(\L) ,
\end{equation}
and so $\red(\L,\rho)$ encodes $\Des(\L)$ along with information about the restriction $\rho$. In Lemma~\ref{lem:ref_des} we will see that  $\des(\L,\rho)$ is a refinement of the information in  $\red(\L,\rho)$.

For a linear order $\L$ on $[p]$, we now always list the values of $\rho$ according to the order $\L$, not according to the natural order on $[p]$. With this convention, $\rhomax{\L}$ is always weakly increasing in the order of $\L$.

\begin{example}\label{ex:reduced-weak-descent}
  Let $\L$ be the linear order $2 \prec 5 \prec 1 \prec 4 \prec 6 \prec 7 \prec 3 \prec 8 \prec 9$ and let $\rho=(\rho(2),\rho(5),\rho(1),\rho(4),\rho(6),\rho(7),\rho(3),\rho(8),\rho(9)) = (2,3,4,3,6,8,6,8,8)$ be a restriction map. The chains after removing edges for descents become
  \[ \overbrace{2 \prec 5}^{C_1} \hspace{2em} \overbrace{1 \prec 4 \prec 6 \prec 7}^{C_2} \hspace{2em} \overbrace{3 \prec 8 \prec 9}^{C_3} \ . \]
  Note that this is the same ascending chain decomposition as in Ex.~\ref{ex:strong-descent} where we computed the descent composition for the same linear order $\L$. We compute $\rhomax{\L}=(2,2,3,3,5,5,6,8,8)$, and so $c_1 = \rhomax{\L}(2) = 2$, $c_2 = \rhomax{\L}(1) = 3$, $c_3 = \rhomax{\L}(3) = 6$. Therefore, $\red(\L,\rho) = (0,|C_1|,|C_2|,0,0,|C_3|) = (0,2,4,0,0,3)$.
  
  The chains after removing edges for $i \precdot j$ and $\rhomax{\L}(i) < \rhomax{\L}(j)$ become
  \[ \overbrace{2 \prec 5}^{C'_1} \hspace{2em} \overbrace{1 \prec 4}^{C'_2} \hspace{2em} \overbrace{6 \prec 7}^{C'_3} \hspace{2em} \overbrace{3}^{C'_4} \hspace{2em} \overbrace{8 \prec 9}^{C'_5} \ . \]
  We have $c'_1 = \rhomax(2) = 2$, $c'_2 = \rhomax(1) = 3$, $c'_3 = \rhomax(6) = 5$, $c'_4 = \rhomax(3) = 6$, and $c'_5 = \rhomax(8) = 8$. Therefore $\des(\L,\rho) = (0,|C'_1|,|C'_2|,0,|C'_3|,|C'_4|,0,|C'_5|) = (0,2,2,0,2,1,0,2)$.
\end{example}

Notice, when $i \precdot j$ and $i>j$, we have $\rhomax{\L}(i) < \rhomax{\L}(j)$. Thus, the following lemma expresses precisely how $\des(\L,\rho)$ refines $\red(\L,\rho)$. 

\begin{lemma}\label{lem:ref_des}
  For $\L$ a linear order on $[p]$ and $\rho: \L \rightarrow \mathbb{N}$ a restriction map, we have $\flatten(\des(\L,\rho))$ refines $\flatten(\red(\L,\rho))$ and $\red(\L,\rho) \trianglerighteq \des(\L,\rho)$.
  \label{lem:red-des}
\end{lemma}

\begin{proof}
  The chains $C_1,\ldots, C_r$ are unions of consecutive chains $C'_1,\ldots, C'_\ell$, and so $(|C'_1|,\ldots,|C'_{\ell}|)$ refines $(|C_1|,\ldots,|C_r|)$, proving the first statement. If $C'_s$ is a sub-chain of $C_t$ for some $s \leq t$, then $c'_s = \rhomax{\L}(\min_\L\{ C'_s\}) \geq \rhomax{\L}(\min_{\L}\{ C_t\}) = c_t$ since $\rhomax{\L}$ is weakly increasing with respect to $\L$. Therefore each nonzero part of $\red(\L,\rho)$ occurs weakly before all of the corresponding nonzero parts of $\des(\L,\rho)$ that refine it, proving the second statement.
\end{proof}

Under certain assumptions on $\L$ and $\rho$, the generating polynomial for $(\L,\rho)$ depends only on the reduced weak descent composition $\red(\L,\rho)$.

The fundamental quasisymmetric functions inspired the \emph{fundamental slide polynomials} introduced by Assaf and Searles \cite{AS17}(Definition~3.6) in the context of Schubert calculus. Fundamental slide polynomials, indexed by weak compositions, form a basis for the full polynomial ring \cite{AS17}(Theorem~3.9). Here, we will say only \emph{slide polynomials} as we do  not use the other polynomials defined~\cite{AS17}.

\begin{definition}[\cite{AS17}]
  For a weak composition $a$ of length $n$, the \emph{slide polynomial} $\fund_{a}$ is given by
  \begin{equation}
    \fund_{a}(x_1,\ldots,x_n) = \sum_{\substack{\flatten(b) \ \mathrm{refines} \ \flatten(a) \\ b \trianglerighteq a}} x_1^{b_1} \cdots x_n^{b_n} .
    \label{e:fundamental-shift}
  \end{equation}
  where the sum is over weak compositions $b$ that dominate $a$ and for which the flattening of $b$ refines the flattening of $a$.
  \label{def:fundamental-shift}
\end{definition}

For example, we have
\begin{displaymath}
  \fund_{(3,0,2)} = x_1^3 x_3^2 + x_1^3 x_2^2 + x_1^3 x_2 x_3 .
\end{displaymath}

Note that if $a$ is a weak composition of length $n$ with the property $a_j>0$ whenever $a_i>0$ for some $i<j$, then $\fund_a = F_{\flatten(a)}(x_1,\ldots,x_n)$ \cite{AS17}(Lemma~3.8). In particular, fundamental quasisymmetric polynomials are slide polynomials. Furthermore, we have the following result from \cite{AS17}(Theorem~4.5).

\begin{theorem}[\cite{AS17}]\label{thm:stable}
  For a weak composition $a$, we have
  \begin{equation}
    \lim_{m\rightarrow\infty}\fund_{0^m \times a}(x_1,\ldots,x_m,0,\ldots,0) = F_{\flatten(a)}(x_1,x_2,\ldots),
  \end{equation}
  where $0^m \times a$ denotes the weak composition obtained by prepending $m$ $0$'s to $a$.
\end{theorem}

With this in mind, we have the following generalization of Proposition~\ref{prop:gessel}.

\begin{proposition}\label{prop:slide}
  Let $\L$ be a linear order on $[p]$ and $\rho: \L \rightarrow \mathbb{N}$ a restriction map. Then $\red(\L,\rho) = \des(\L,\rho)$ if and only if for all $i \precdot j$, whenever $i < j$ we have $\rhomax{\L}(i) = \rhomax{\L}(j)$. Moreover, when this is the case, we have
  \begin{equation}
    \fund_{(\L,\rho)} = \fund_{\red(\rho)} .
  \end{equation}
  Conversely, for any weak composition $a$, there exists $\rho$ and $\L$ such that $\fund_{(\L,\rho)} = \fund_{a}$.
\end{proposition}

\begin{proof}
  We may assume $\rho = \rhomax{\L}$ and let $\comp(\rho)=\red(\rho)$. The condition on $\rho$ ensures the chains used to compute $\des(\L,\rho)$ are broken only between descents, since we never have $i \precdot j$, $i<j$, and $\rhomax{\L}(i) < \rhomax{\L}(j)$. Thus $\des(\L,\rho)=\red(\L,\rho)$.

  Let $f \in \A_{\rho}(\L)$ be a $(\L,\rho)$-partition and let $b=\comp(f)$. Then $b \trianglerighteq \comp(\rho)$. Letting $\L = \{\ell_1 \prec \ell_2 \prec \cdots \prec \ell_p \}$ so that we have
  \begin{displaymath}
    \begin{array}{ccccccc}
      f(\ell_1) & \leq & f(\ell_2) & \leq & \cdots & \leq & f(\ell_p) \\
      \rotatebox[origin=c]{90}{$\geq$} & &
      \rotatebox[origin=c]{90}{$\geq$} & & & &
      \rotatebox[origin=c]{90}{$\geq$} \\      
      \rho(\ell_1) & \leq & \rho(\ell_2) & \leq & \cdots & \leq & \rho(\ell_p)
    \end{array}
  \end{displaymath}
  notice $\ell_i < \ell_{i+1}$ implies $\rho(\ell_i) = \rho(\ell_{i+1})$, and $\ell_i > \ell_{i+1}$ implies both $f(\ell_i) < f(\ell_{i+1})$ and $\rho(\ell_i) < \rho(\ell_{i+1})$. Thus $\flatten(b)$ refines $\flatten(\comp(\rho))$. In particular, every $(\L,\rho)$-partition contributes a term of $\fund_{\comp(\rho)}=\fund_{\red(\rho)}$. Conversely, for any weak composition $b$ such that $\flatten(b)$ refines $\flatten(\comp(\rho))$ and $b \trianglerighteq \comp(\rho)$, we may construct a unique $(\L,\rho)$-partition $f$ such that $\comp(f)=b$.

  Given $a$ with $\flatten(a) = (\alpha_1,\ldots,\alpha_k)$, we may take $\L$ to be the linear order
  \[ (p-\alpha_{1}+1) \prec \cdots \prec p \prec (p - \alpha_{1} - \alpha_2 + 1) \prec \cdots \prec (p-\alpha_{1}) \prec \cdots \prec 1 \prec \cdots \prec \alpha_k . \]
  Then clearly $\Des(\L) = \alpha$. If the nonzero entries of $a$ occur at indices $r_1 < \cdots < r_k$, then we may take $\rho$ to be the restriction map such that $\rho(i) = r_j$ whenever $i$ is in the $j$th descent chain of $\L$. This ensures $a = \comp(\rho) = \red(\L,\rho) = \des(\L,\rho)$, and so, by the first statement of the proposition, $\fund_{(\L,\rho)} = \fund_{a}$ as desired.  
\end{proof}

\begin{example}\label{ex:realize} Given the weak composition $a = (0,2,4,0,0,3)$, we may realize the pair $(\L,\rho)$ for which $\fund_a = \fund_{(\L,\rho)}$ as follows. We have $\flatten(a) = (2,4,3)$ and so we will take $\L$ to be the linear order
  \[ \overbrace{8 \prec 9}^{C_1} \prec \overbrace{4 \prec 5 \prec 6 \prec 7}^{C_2} \prec \overbrace{1 \prec 2 \prec 3}^{C_3} . \]
  The nonzero entries of $a$ occur at indices $2,3,6$, and so we set $\rho$ to be
  \[ \rho(8) = \rho(9) = 2, \hspace{1em} \rho(4) = \rho(5) = \rho(6) = \rho(7) = 3, \hspace{1em} \rho(1) = \rho(2) = \rho(3) = 6 . \] 
\end{example}

Observe that the conclusion of Proposition~\ref{prop:slide} does not hold in general. 

\begin{example}\label{ex:linnonbasis} Take $\L$ to be $2 \prec 3 \prec 1$ and $\rho = (2,3,4)$. Then we have
  \begin{eqnarray*}
    \fund_{(\L,\rho)} & = & x_1^2 x_2 + x_1^2 x_3 + x_1^2 x_4 + x_1 x_2 x_3 + x_1 x_2 x_4 \\
    & & + x_1 x_3 x_4 + x_2^2 x_3 + x_2^2 x_4 + x_2 x_3 x_4 \\
    & = & \fund_{(0,2,0,1)} + \fund_{(0,1,1,1)} - \fund_{(1,1,0,1)}.
  \end{eqnarray*}
  In this case, $\rhomax{\L}=\rho$ but it is not constant over the increasing chain $2 \prec 3$. Here, $\red(\L,\rho) = (0,2,0,1)$, and indeed we see $\fund_{\red(\L,\rho)}$ appear as a term in $\fund_{(\L,\rho)}$, but there are other terms as well. Note as well the slide expansion is not positive.
\end{example}

\begin{remark}\label{rem:bouded_quasi}
  In Remark~\ref{rem:quasi} we saw that the linear extensions of $\P$-partitions span an algebra containing all $\P$-partition generating functions, and this algebra has a basis consisting of linear extensions with distinct descent compositions. Corollary~\ref{cor:fundamental} gives us a beautiful positive expansion of any $\P$-partition generating functions in terms of the fundamental basis.
  We aim to obtain similar results for $(\P,\rho)$-partitions, but already Ex~\ref{ex:linnonbasis} shows that we cannot expect such result for all $(\P,\rho)$. While Corollary~\ref{cor:ext} gives a direct analogue of Corollary~\ref{cor:fundamental} expanding positively any bounded $\P$-partitions in term of linear bounded $\P$-partition, the example shows that, unfortunately, the linear bounded $\P$-partition are not linearly independent nor positive in term of slide polynomials. In Section~\ref{sec:flagged} we will define flagged $(\P,\rho)$-partitions for which all the desired properties hold.
\end{remark}

\subsection{Fundamental Theorem for $(\P,\rho)$-partitions}
\label{sec:FTRPP}

Stanley's Fundamental Theorem of $\P$-partitions holds for $(\P,\rho)$-partitions as well.

\begin{theorem}[Fundamental Theorem of  $(\P,\rho)$-partitions]\label{lem:restricted-fundamental}
  Given a poset $\P$ on $[p]$ and a restriction $\rho$, we have
  \begin{equation}\label{e:restricted-fundamental}
    \A_{\rho}(\P) = \bigsqcup_{\L \in \mathscr{L}(\P)} \A_{\rho}(\L),
  \end{equation}
  where the disjoint union is over the set $\mathscr{L}(\P)$ of linear extensions of $\P$. 
\end{theorem}

\begin{proof}
  Without lost of generality, we may assume that $\rho=\rhomax{\P}$. We proceed by induction on the number of incomparable pairs $i,j$ in $\P$. If there is no such pair, then $\P$ must be a linear order, and so $\mathscr{L}(\P)=\{\P\}$ and the result follow trivially. Now assume that $\P$ has some incomparable pair. Fix $i<j$  an incomparable pair in $\P$. We construct two new posets $\P_{i\prec j}$ and $\P_{i\succ j}$ that are the transitive closure obtained by adding $i\prec j$ or  $i\succ j$ to $\P$, respectively. We will show that
  \begin{equation}\label{eq:Pbasic}
    \A_{\rho}(\P) =  \A_{\rho}(\P_{i\prec j}) \sqcup   \A_{\rho}(\P_{i\succ j}),
  \end{equation}
  a disjoint union. The result will follow from the induction hypothesis.
	
  First note that $\A_{\rho}(\P_{i\prec j}) \cap   \A_{\rho}(\P_{i\succ j})=\varnothing$. Indeed, a $(\P,\rho)_{i\prec j}$-partition $f$ satisfies $f(i)\le f(j)$ whereas a $(\P,\rho)_{i\succ j}$-partition $f$ satisfies $f(i)> f(j)$. Since any relation of $\P$ is contained in both $\P_{i\prec j}$ and $\P_{i\succ j}$, we have that if $f\in  \A_{\rho}(\P_{i\prec j}) \sqcup   \A_{\rho}(\P_{i\succ j})$, then $f\in \A_{\rho}(\P)$. The restriction $\rho$ is the same function across the three sets involved, so if has the same $\rho$-restrictions.

  Now assume you have $f\in  \A_{\rho}(\P)$ and consider the values $f(i)$ and $f(j)$. If $f(i)\le f(j)$, then $f\in  \A_{\rho}(\P_{i\prec j})$. Indeed it satisfies all the conditions imposed by $\P$ plus the additional condition imposed by $i\prec j$. Moreover the restriction imposed by $\rho$ is the same on both sets. Now if $f(i)> f(j)$, then $f\in  \A_{\rho}(\P_{i\succ j})$. Such $f$ satisfies all conditions imposed by $\P$ plus the condition imposed by $i\succ j$ and $i<j$. This concludes our proof.
\end{proof}

In particular, this gives a simple decomposition of the generating function of a poset as the sum of generating functions of all linear extensions of the poset.

\begin{corollary}\label{cor:ext}
  Given a poset $\P$ on $[p]$ and a restriction $\rho$, we have
  \begin{equation}\label{e:rectricted-Fundamental}
    \fund_{(\P,\rho)} = \sum_{\L \in \mathscr{L}(\P)} \fund_{(\L,\rho)}.
  \end{equation}  
\end{corollary}

\begin{remark}
 The above corollary gives an efficient algorithm to program the polynomials $\fund_{(\P,\rho)}$, especially in packages like MathSage that
 contains modules for orders and linear extension. But again, be aware that this does not, in general, give a decomposition into the slide polynomial basis.
\end{remark}

%
\section{Flagged $(\P,\rho)$-partitions}
%
\label{sec:flagged}

Assaf and Searles \cite[Theorem~3.9]{AS17} showed that fundamental slide polynomials form a basis of polynomials, and Proposition~\ref{prop:slide} shows that slides polynomial are $(\L,\rho)$-partitions for some $\rho$ and $\L$. However, Ex.~\ref{ex:linnonbasis} shows that for some restrictions $\rho$ on linear orders we do not get single slide polynomials, and, in fact, the expansion in terms of slides polynomials may not be positive. In this section we are interested in an important subfamilies of  $(\P,\rho)$-partitions that will give positivity results when expanded in the basis of slide polynomials.

\subsection{Flagged $(\P,\rho)$-partitions} 
\label{sec:mainproof}

The subfamily of $(\P,\rho)$-partitions of interest includes the \emph{flagged Schur functions} defined by Lascoux and Sch\"{u}tzenberger~\cite{LS82} and studied further by Wachs~\cite{wachs82}. This motivates calling the elements of this subfamily \emph{flagged} $(\P,\rho)$-partitions.

\begin{definition}\label{def:P-flag}
 For a poset $\P$, a restriction $\rho$ is a \emph{$\P$-flag} if $\rho(i)>0$ for all $i\in[p]$,  
  \begin{enumerate}
  \item[(i)] if $i \precdot j$ and $i<j$, then $\rho(i) = \rho(j)$, and
  \item[(ii)] if $i \precdot j$ and $i > j$, then $\rho(i) \le \rho(j)$.
  \end{enumerate}
  We call $\A_\rho(\P)$ the set of \emph{flagged $(\P,\rho)$-partitions} to emphasize that $\rho$ is a $\P$-flag.
\end{definition}

We see in Proposition~\ref{prop:slide} that the slide polynomials are exactly the generating functions of flagged $(\L,\rho)$-partitions for $\L$ a linear order.
To prove our main theorem, we introduce a second family of $(\P,\rho)$-partitions that we call \emph{well-labelled}.

\begin{definition}\label{def:well-label}
 For a poset $\P$ on [p] and a restriction $\rho$, we say that $(\P,\rho)$ is \emph{well-labelled}  if
  \begin{enumerate}
  \item[(i)] if $i \precdot j$ and $i<j$, then $\rho(i) \ge \rho(j)$, 
  \item[(ii)] if $i \precdot j$ and $i > j$, then $\rho(i) \le \rho(j)$, and
  \item[(iii)] for all incomparable $i,j\in \P$, if $ i<j$, then $\rho(i) \geq \rho(j)$.
  \end{enumerate}
\end{definition}

The following result shows that every $\P$-flag $\rho$ generates a polynomial equivalent to a polynomial obtained by a well-labelled pair $(\P',\rho')$.

\begin{proposition}\label{prop:well}
  Given poset $\P$ on [p] and a $\P$-flag $\rho$, there exists a permutation $\pi\colon [p]\to[p]$ such that 
  \begin{enumerate}
  \item  $\rho\circ\pi$  is a $\pi^{-1}(\P)$-flag,
  \item  $\big(\pi^{-1}(\P),\rho\circ\pi\big)$ is well-labelled, and
  \item  $\A_{\rho}(\P)=\A_{\rho\circ\pi}(\pi^{-1}(\P))$.
  \end{enumerate}
\end{proposition}

\begin{proof}
  The Hasse diagram $H(\P)$ of $\P$ is the directed graph on $[p]$ where we have an edge  $i\to j$ for each cover $i\precdot j$. Consider the connected component $M_1, M_2, \ldots , M_\ell$ of the directed graph $H^+(\P)$ obtained from $H(\P)$ by removing every edges $i\precdot j$ where $i>j$. The order $\P$ induces a partial order on these connected components, and since $\rho$ is a $\P$-flag, it must be constant on each connected component. Therefore we may re-index the components as follows. Setting $m_i=\rho(x)$ for any $x\in M_i$, we may assume the $M_i$'s are indexed in such a way that $m_1\ge m_2\ge\cdots\ge m_\ell$ and $M_\ell\cdots M_2M_1$ is a linear extension of the $\P$-induced order on the $M_i$'s. 

  Define a permutation $\pi$ on $[p]$ such that for $1 \leq k \leq \ell$ we have
  \begin{itemize}
  \item[(a)] $\{\pi(|M_1|+\cdots +|M_{k-1}|+1), \ldots, \pi(|M_1|+\cdots +|M_{k}|)\} = M_k$, and
  \item[(b)] $\pi(|M_1|+\cdots +|M_{k-1}|+1) < \cdots < \pi(|M_1|+\cdots +|M_{k}|)$.
  \end{itemize}
  We claim $\pi$ has the desired properties. To see that $\rho\circ\pi$ is a $\pi^{-1}(\P)$-flag, suppose $i \precdot_{\pi^{-1}(\P)} j$ or, equivalently, $\pi(i) \precdot_{\P} \pi(j)$. If $\pi(i)<\pi(j)$, then $\pi(i)$ and $\pi(j)$ are the same component $M_k$, and so we must have $i<j$ as well by (a) and (b). Conversely, if $\pi(i)>\pi(j)$, then since $\rho$ is a $\P$-flag, we have $\rho(\pi(i)) \le \rho(\pi(j))$ placing $\pi(i)$ in a later component than $\pi(j)$ by our indexing choice, and so $i>j$. 
  
  To see that $\pi^{-1}(\P)$ is well-labelled by $\rho\circ\pi$, notice conditions (i) and (ii) of Definition~\ref{def:P-flag} imply conditions (i) and (ii) of Definition~\ref{def:well-label}. For (iii), suppose $i$ and $j$ are incomparable in $\pi^{-1}(\P)$ with $i<j$. If $\pi(i),\pi(j)$ are in the same component $M_k$, then  $\rho\circ\pi(i) = \rho\circ\pi(j)$. Otherwise, since $i<j$, we must have $\pi(i)$ in an earlier, and thus greater, component than $\pi(j)$, and so $\rho\circ\pi(i) \ge \rho\circ\pi(j)$, as desired.

  Finally, to see that the partition generating functions coincide, regarding the nodes of $\P$ as fixed, as we permute their labels with $\pi^{-1}$, we correspondingly permute the restrictions so that the latter remain fixed with each node. Thus the set of $(\P,\rho)$-partitions is identical to that of $(\pi^{-1}(\P),\rho\circ\pi)$-partitions.  
\end{proof}

\begin{example}\label{ex:orderings}
  Consider the poset $\P$ in Fig.~\ref{fig:orderingMs}. The connected components of $H^+(\P)$ are $M_1=\{1,2,5\}$, $M_2=\{3,6\}$, $M_3=\{4\}$ and $M_4=\{7\}$. The  $\P$-flag $\rho$ is constant on each component with values $m_1=4$, $m_2=4$, $m_3=2$ and $m_4=2$. The indexing is valid since $4\ge 4 \ge 2\ge 2$ and  $M_4M_3M_2M_1$ is a linear extension of the $\P$-induced order on the components. The permutation $\pi$ is given by
  \[ \pi = \left(\begin{array}{ccccccc}
     1 & 2 & 3 & 4 & 5 & 6 & 7 \\
     1 & 2 & 5 & 3 & 6 & 4 & 7
  \end{array} \right) \]
  with result that $\pi^{-1}(\P)$ is well-labelled by $\rho \circ \pi$, as seen in Fig.~\ref{fig:orderingMs}.
\end{example}

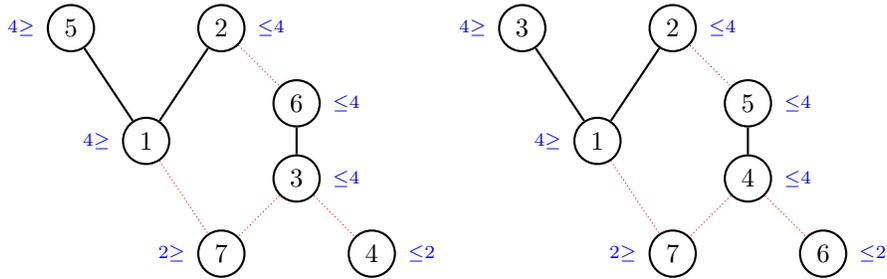
\begin{figure}[ht]
  \begin{center}
    \begin{tikzpicture}[
        roundnode/.style={circle, draw=black, thick, minimum size=1ex},
        label/.style={%
          postaction={ decorate,transform shape,
            decoration={ markings, mark=at position .5 with \node #1;}}}]
      \node[roundnode] at (0,3)   (P5) {$5$};
      \node[roundnode] at (1,1.5) (P1) {$1$};
      \node[roundnode] at (2,0)   (P7) {$7$};
      \node[roundnode] at (2,3)   (P2) {$2$};
      \node[roundnode] at (3,1)   (P4) {$3$};
      \node[roundnode] at (3,2)   (P6) {$6$};
      \node[roundnode] at (4,0)   (P3) {$4$};
      \node[left  = 1pt of P1] {\Blue{$\scriptstyle 4\geq$}};
      \node[right = 1pt of P2] {\Blue{$\scriptstyle\leq 4$}};
      \node[right = 1pt of P3] {\Blue{$\scriptstyle\leq 2$}};
      \node[right = 1pt of P4] {\Blue{$\scriptstyle\leq 4$}};
      \node[left  = 1pt of P5] {\Blue{$\scriptstyle 4\geq$}};
      \node[right = 1pt of P6] {\Blue{$\scriptstyle\leq 4$}};
      \node[left  = 1pt of P7] {\Blue{$\scriptstyle 2\geq$}};
      \draw[thick] (P5) -- (P1) ;
      \draw[thick] (P1) -- (P2) ;
      \draw[thin,densely dotted,red] (P4) -- (P3) ;
      \draw[thick] (P6) -- (P4) ;
      \draw[thin,densely dotted,red] (P2) -- (P6) ;
      \draw[thin,densely dotted,red] (P1) -- (P7) ;
      \draw[thin,densely dotted,red] (P4) -- (P7) ;
      \node[roundnode] at (6,3)   (Q5) {$3$};
      \node[roundnode] at (7,1.5) (Q1) {$1$};
      \node[roundnode] at (8,0)   (Q7) {$7$};
      \node[roundnode] at (8,3)   (Q2) {$2$};
      \node[roundnode] at (9,1)   (Q4) {$4$};
      \node[roundnode] at (9,2)   (Q6) {$5$};
      \node[roundnode] at (10,0)   (Q3) {$6$};
      \node[left  = 1pt of Q1] {\Blue{$\scriptstyle 4\geq$}};
      \node[right = 1pt of Q2] {\Blue{$\scriptstyle\leq 4$}};
      \node[right = 1pt of Q3] {\Blue{$\scriptstyle\leq 2$}};
      \node[right = 1pt of Q4] {\Blue{$\scriptstyle\leq 4$}};
      \node[left  = 1pt of Q5] {\Blue{$\scriptstyle 4\geq$}};
      \node[right = 1pt of Q6] {\Blue{$\scriptstyle\leq 4$}};
      \node[left  = 1pt of Q7] {\Blue{$\scriptstyle 2\geq$}};
      \draw[thick] (Q5) -- (Q1) ;
      \draw[thick] (Q1) -- (Q2) ;
      \draw[thin,densely dotted,red] (Q4) -- (Q3) ;
      \draw[thick] (Q6) -- (Q4) ;
      \draw[thin,densely dotted,red] (Q2) -- (Q6) ;
      \draw[thin,densely dotted,red] (Q1) -- (Q7) ;
      \draw[thin,densely dotted,red] (Q4) -- (Q7) ;
    \end{tikzpicture}
  \end{center}
  \caption{\label{fig:orderingMs} An example of a poset $\P$ together with a $\P$-flag $\rho$ (left) and the poset $\pi^{-1}(\P)$ well-labelled by $\rho\circ\pi$ (right).}
\end{figure}

In view of Proposition~\ref{prop:well}, to prove our Main Theorem~\ref{thm:main}, we can always assume that we have a well-labelled flagged $(\P,\rho)$. The next result is the heart of the proof; it relies on the Fundamental Theorem~\ref{lem:restricted-fundamental} for $(\P,\rho)$-partitions.

\begin{lemma}\label{lem:flag_decomposition}
  Let $\P$ be a poset on $[p]$ and $\rho$ a map such that $(\P,\rho)$ is well-labelled. For any incomparable pair $i,j\in[p]$, let $\P_1 = \P_{i\prec j}$ and $\P_2 = \P_{i\succ j}$ be the transitive closure obtained by adding $i\prec j$ or  $i\succ j$ to $\P$, respectively. Then $(\P_1,\rho)$ and $(\P_2,\rho)$ are both well-labelled. 
\end{lemma}

\begin{proof}
Since $i$ and $j$ are incomparable in $\P$ with $i<j$, and $(\P,\rho)$ is well-labeled, we must have $\rho(i) \geq \rho(j)$. 
We claim $i \precdot j$ in $\P_1 = \P_{i\prec j}$, since otherwise any intermediate $k$ would already exist in $\P$, contradicting the incomparability of $i$ and $j$. 
Therefore $i \precdot j$, $i<j$ and $\rho(i) \geq \rho(j)$, hence condition (i) of Definition~\ref{def:well-label} is satisfied. For all other pairs of elements
in $\P_1$ the conditions of being well labelled are already satisfied from $\P$. For $\P_2$, the situation is similar, we have $j\precdot i$, $j>i$ and $\rho(j) \le \rho(i)$, and condition (ii) of Definition~\ref{def:well-label} is satisfied. It  is the only pair we need to check and the rest follows.
\end{proof}

\begin{lemma}\label{lem:well-labelledL}
  Let $\L$ be a linear order on $[p]$ and $\rho$ a map such that $(\L,\rho)$ is well-labelled. 
  We have that $\fund_{(\L,\rho)}=\fund_{(\L,\rhomax{\L})}$ is a slide polynomial or zero.
\end{lemma}

\begin{proof}
  If $\rhomax{\L}$ has any values $\le 0$, then we have $\fund_{(\L,\rhomax{\L})}=0$. 
  Now,  since $\rho$ is weakly decreasing on segments of $\L$ that do not have descents, $\rhomax{\L}$ will be constant on those same segment. The lemma then follow from Proposition~\ref{prop:slide}. 
\end{proof}

\begin{proof}[Proof of Main Theorem~\ref{thm:main}]
  Given a poset $\P$ on $[p]$ and a $\P$-flag $\rho$, by Proposition~\ref{prop:well} we may transform this to a well-labelled and flagged $(\P',\rho')$ such that
  $$\fund_{(\P,\rho)}=\fund_{(\P',\rho')}.$$
  We then use  the Fundamental Theorem~\ref{lem:restricted-fundamental} for $(\P,\rho)$-partitions.  In the proof of Theorem~\ref{lem:restricted-fundamental}, we recursively use Equation~\eqref{eq:Pbasic}. Lemma~\ref{lem:flag_decomposition} guaranties that at each step, we get well-labelled
  pair. Hence by induction on the number of incomparable pairs we get that
  $$\fund_{(\P,\rho)}=\fund_{(\P',\rho')}=\sum_{\L \in \mathscr{L}(\P)} \fund_{(\L,\rho)}$$
  with each pair $(\L,\rho)$ in the sum well-labelled. Therefore Theorem~\ref{thm:main} follows from Lemma~\ref{lem:well-labelledL}.
\end{proof}

\begin{example}\label{ex:slide-P}
  Let $\P$ be the poset depicted on the left side of Fig.~\ref{fig:slide2}, and let $\rho$ be the restriction map $(\rho(1),\rho(2),\rho(3)) = (4,3,2)$. 
  The pair $(\P,\rho)$ is well labelled and flagged.
  Take two incomparable elements in $\P$, say $1$ and $2$. We add the relation $1\prec 2$ or 
$2\prec 1$ obtaining the order $\P_1$ and $\P_2$ in Fig.~\ref{fig:slide2}. We have  that $(\P_1,\rho)$ and $(\P_2,\rho)$ are well labelled.
We can repeat the process for  $\P_2$ and get two well labelled pairs $(\L_1,\rho)$ and $(\L_2,\rho)$.
The order $\P_1$, $\L_1$ and $\L_2$ are linear order, we compute $\rhomax{\L}$ for each
(see Figure~\ref{fig:slide3}). We get that
  \[ \fund_{(\P,\rho)} = \fund_{(0,1,2,0)} + \fund_{(0,1,1,1)} + \fund_{(0,2,0,1)} ,\]
  where the right hand side are slide polynomials. 
  \end{example}

\begin{figure}[ht]
  \begin{center}
    \begin{tikzpicture}[
        roundnode/.style={circle, draw=black, thick, minimum size=1ex},
        label/.style={%
          postaction={ decorate,transform shape,
            decoration={ markings, mark=at position .5 with \node #1;}}}]
      \node at (-0.5,1) {$\P$};
      \node[roundnode] at (0,2)   (P1) {$1$};
      \node[roundnode] at (0,0)   (P2) {$3$};
      \node[roundnode] at (1,1)   (P3) {$2$};
      \node[right = 1pt of P1] {\Blue{$\scriptstyle\leq 4$}};
      \node[right = 1pt of P2] {\Blue{$\scriptstyle\leq 2$}};
      \node[right = 1pt of P3] {\Blue{$\scriptstyle\leq 3$}};
      \draw[thin] (P1) -- (P2) ;
      \node at (3.5,1) {$\P_1:$};
      \node[roundnode] at (4.5,1)   (L1) {$1$};
      \node[roundnode] at (4.5,0)   (L2) {$3$};
      \node[roundnode] at (4.5,2)   (L3) {$2$};
      \node[right = 1pt of L1] {\Red{\xcancel{\Blue{$\scriptstyle\leq 4$}} $\scriptstyle\leq 3$}};
      \node[right = 1pt of L2] {\Blue{$\scriptstyle\leq 2$}};
      \node[right = 1pt of L3] {\Blue{$\scriptstyle\leq 3$}};
      \draw[thin] (L3) -- (L1) ;      
      \draw[thin] (L1) -- (L2) ;      
      \node at (7,1) {$\P_2:$};
      \node[roundnode] at (8.5,2)   (PP1) {$1$};
      \node[roundnode] at (7.5,0)   (PP2) {$3$};
      \node[roundnode] at (9.5,0)   (PP3) {$2$};
      \node[right = 1pt of PP1] {\Blue{$\scriptstyle\leq 4$}};
      \node[right = 1pt of PP2] {\Blue{$\scriptstyle\leq 2$}};
      \node[right = 1pt of PP3] {\Blue{$\scriptstyle\leq 3$}};
      \draw[thin] (PP1) -- (PP3) ;      
      \draw[thin] (PP1) -- (PP2) ;      
    \end{tikzpicture}
  \end{center}
  \caption{\label{fig:slide2} An example of decomposition $ \A_{\rho}(\P) =  \A_{\rho}(\P_1) \sqcup   \A_{\rho}(\P_2)$ for
 well-labelled flagged $(\P,\rho)$-partition. The pairs $(\P_1,\rho)$  and  $(\P_2,\rho)$ are well labelled.}
\end{figure}
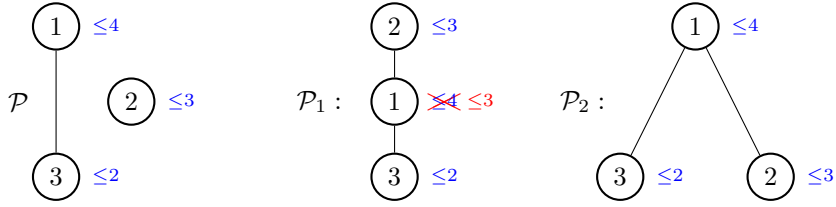

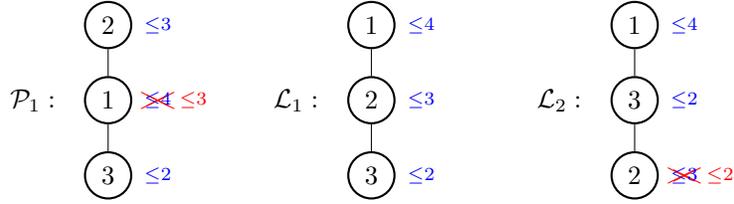
\begin{figure}[ht]
  \begin{center}
    \begin{tikzpicture}[
        roundnode/.style={circle, draw=black, thick, minimum size=1ex},
        label/.style={%
          postaction={ decorate,transform shape,
            decoration={ markings, mark=at position .5 with \node #1;}}}]
      \node at (0.5,1) {$\P_1:$};
      \node[roundnode] at (1.5,1)   (L1) {$1$};
      \node[roundnode] at (1.5,0)   (L2) {$3$};
      \node[roundnode] at (1.5,2)   (L3) {$2$};
      \node[right = 1pt of L1] {\Red{\xcancel{\Blue{$\scriptstyle\leq 4$}} $\scriptstyle\leq 3$}};
      \node[right = 1pt of L2] {\Blue{$\scriptstyle\leq 2$}};
      \node[right = 1pt of L3] {\Blue{$\scriptstyle\leq 3$}};
      \draw[thin] (L3) -- (L1) ;      
      \draw[thin] (L1) -- (L2) ;          
      \node at (4,1) {$\L_1:$};
      \node[roundnode] at (5,2)   (LL1) {$1$};
      \node[roundnode] at (5,0)   (LL2) {$3$};
      \node[roundnode] at (5,1)   (LL3) {$2$};
      \node[right = 1pt of LL1] {\Blue{$\scriptstyle\leq 4$}};
      \node[right = 1pt of LL2] {\Blue{$\scriptstyle\leq 2$}};
      \node[right = 1pt of LL3] {\Blue{$\scriptstyle\leq 3$}};
      \draw[thin] (LL1) -- (LL3) ;      
      \draw[thin] (LL3) -- (LL2) ;      
      \node at (7.5,1) {$\L_2:$};
      \node[roundnode] at (8.5,2)   (LLL1) {$1$};
      \node[roundnode] at (8.5,1)   (LLL2) {$3$};
      \node[roundnode] at (8.5,0)   (LLL3) {$2$};
      \node[right = 1pt of LLL1] {\Blue{$\scriptstyle\leq 4$}};
      \node[right = 1pt of LLL2] {\Blue{$\scriptstyle\leq 2$}};
      \node[right = 1pt of LLL3] {\Red{\xcancel{\Blue{$\scriptstyle\leq 3$}} $\scriptstyle\leq 2$}};
      \draw[thin] (LLL1) -- (LLL2) ;      
      \draw[thin] (LLL2) -- (LLL3) ;      
    \end{tikzpicture}
  \end{center}
  \caption{\label{fig:slide3} Terminating the example, $\rhomax{\L}$ is computed (in red).}
\end{figure}

\begin{remark}
  If we remove the hypothesis that $(\P,\rho)$ is well-labelled in Lemma~\ref{lem:flag_decomposition}, then the statement is  wrong and cannot be fixed easily. For example looking at Figure~\ref{fig:badlabel}, $(\P'',\rho)$ is not well-labelled as $2,3$ are incomparable in $\P''$, but $2<3$ and $\rho(2)=2<3=\rho(3)$. It is true that
  $$\fund_{(\P'',\rho)} = \fund_{(\L''_1,\rhomax{\L''_1})} +  \fund_{(\L''_2,\rhomax{\L''_2})},$$
  but $(\L''_1,\rhomax{\L''_1})$ is not well labelled and in fact $ \fund_{(\L''_1,\rhomax{\L''_1})}$ is not a single slide polynomial. The slide polynomial $\fund_{(\L''_2,\rhomax{\L''_2})}=\fund_{(1,1,0,1)}$ will cancel a term in $ \fund_{(\L''_1,\rhomax{\L''_1})}$ to give the correct expansion 
  $$\fund_{(\P'',\rho)}=\fund_{(\P_2,\rho)}=\fund_{(0,1,1,1)}+\fund_{(0,2,0,1)}$$
  as computed with the well-labelled flagged pair $(\P_2,\rho)$ decomposing as $(\L_1,\rho)$ and $(\L_2,\rho)$ in Figure~\ref{fig:slide3}.
\end{remark}

\begin{figure}[ht]
  \begin{center}
    \begin{tikzpicture}[
        roundnode/.style={circle, draw=black, thick, minimum size=1ex},
        label/.style={%
          postaction={ decorate,transform shape,
            decoration={ markings, mark=at position .5 with \node #1;}}}]
      \node at (-0.5,1) {$\P'':$};
      \node[roundnode] at (1,2)   (PP1) {$1$};
      \node[roundnode] at (0,0)   (PP2) {$2$};
      \node[roundnode] at (2,0)   (PP3) {$3$};
      \node[right = 1pt of PP1] {\Blue{$\scriptstyle\leq 4$}};
      \node[right = 1pt of PP2] {\Blue{$\scriptstyle\leq 2$}};
      \node[right = 1pt of PP3] {\Blue{$\scriptstyle\leq 3$}};
      \draw[thin] (PP1) -- (PP3) ;      
      \draw[thin] (PP1) -- (PP2) ;      
      \node at (4,1) {$\L''_1:$};
      \node[roundnode] at (5,2)   (LL1) {$1$};
      \node[roundnode] at (5,0)   (LL2) {$2$};
      \node[roundnode] at (5,1)   (LL3) {$3$};
      \node[right = 1pt of LL1] {\Blue{$\scriptstyle\leq 4$}};
      \node[right = 1pt of LL2] {\Blue{$\scriptstyle\leq 2$}};
      \node[right = 1pt of LL3] {\Blue{$\scriptstyle\leq 3$}};
      \draw[thin] (LL1) -- (LL3) ;      
      \draw[thin] (LL3) -- (LL2) ;      
      \node at (7.5,1) {$\L''_2:$};
      \node[roundnode] at (8.5,2)   (LLL1) {$1$};
      \node[roundnode] at (8.5,1)   (LLL2) {$2$};
      \node[roundnode] at (8.5,0)   (LLL3) {$3$};
      \node[right = 1pt of LLL1] {\Blue{$\scriptstyle\leq 4$}};
      \node[right = 1pt of LLL2] {\Blue{$\scriptstyle\leq 2$}};
      \node[right = 1pt of LLL3] {\Red{\xcancel{\Blue{$\scriptstyle\leq 3$}} $\scriptstyle\leq 1$}};
      \draw[thin] (LLL1) -- (LLL2) ;      
      \draw[thin] (LLL2) -- (LLL3) ;      
    \end{tikzpicture}
  \end{center}
  \caption{\label{fig:badlabel} An example where $\P''$ is not well-labelled with respect to $\rho$ and Lemma~\ref{lem:flag_decomposition} fails.}
\end{figure}
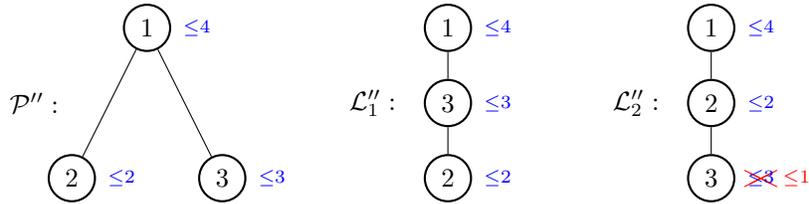

In particular, when $\rho$ is a $\P$-flag, we have the following analog of Corollary~\ref{cor:fundamental} that gives a combinatorial formula for the nonnegative slide expansion.

\begin{corollary}\label{cor:fund}
  Given a poset $\P$ on $[p]$ and a (well labelled) $\P$-flag $\rho$, we have
  \begin{equation}\label{e:fund-fund}
    \fund_{(\P,\rho)} = \sum_{\L \in \mathscr{L}(\P)} \fund_{(\L,\rho)} = \sum_{\L \in \mathscr{L}(\P)} \fund_{\red(\L,\rho)}.
  \end{equation}  
\end{corollary}

\subsection{Two applications}

We now give two important examples of positive expansions in term of slide polynomials that are analogues of Corollary~\ref{cor:shuffle} and Corollary~\ref{cor:syt}, respectively.

\subsubsection{Product of slide polynomials}
In Corollary~\ref{cor:shuffle}, the theory of $\P$-partitions can be used to determine the structure constants for fundamental quasisymmetric functions in terms of the shuffle product on strong compositions. Assaf and Searles \cite{AS17}(Definition~5.9) generalized the shuffle product on strong compositions to the \emph{slide product} on weak compositions.

\begin{definition}[\cite{AS17}]\label{def:slide}
    Let $a,b$ be weak compositions of length $n$. Let $A$ and $B$ be words defined by $A = (2n-1)^{a_1} \cdots (3)^{a_{n-1}} (1)^{a_n}$ and $B = (2n)^{b_1} \cdots (4)^{b_{n-1}} (2)^{b_n}$. Define the \emph{shuffle set of $a$ and $b$}, denoted by $\mathrm{Sh}(a,b)$, by
  \begin{equation}
    \mathrm{Sh}(a,b) = \{ C \in A \shuffle B \mid \Des_A(C) \geq a \mbox{ and } \Des_B(B) \geq b \},
  \end{equation}
  where $\Des_A(C)_i$ (respectively $\Des_B(C)_i$) is the number of letters from $A$ (respectively $B$) in the $i$th increasing run of $C$.

  Define the \emph{slide product of $a$ and $b$}, denoted by $a \shuffle b$, to be the formal sum
  \begin{equation}
    a \shuffle b =  \sum_{C \in \mathrm{Sh}(a,b)} \Des(\mathrm{bump}_{(a,b)}(C))
  \end{equation}
  where $\mathrm{bump}_{(a,b)}(C)$ is the unique element of the shuffle set $0^{n-\ell(\Des(C))} \shuffle C$ such that $\Des_A(\mathrm{bump}_{(a,b)}(C)) \geq a$ and $\Des_B(\mathrm{bump}_{(a,b)}(C)) \geq b$ and if $D \in 0^{n-\ell} \shuffle C$ satisfies $\Des_A(D) \geq a$ and $\Des_B(D) \geq b$, then $\Des(D) \geq \Des(\mathrm{bump}_{(a,b)}(C))$. 
\end{definition}

Assaf and Searles then used the slide product to characterize the structure constants for fundamental slide polynomials \cite[Theorem~5.11]{AS17} as follows.

\begin{theorem}[\cite{AS17}]\label{thm:slide}
  Let $a, b$ be two weak compositions. Then
  \begin{equation}\label{e:slide}
    \fund_{a} \fund_{b} = \sum_{c} [c \mid a \shuffle b] \fund_{c},
  \end{equation}
  where $[c \mid a \shuffle b]$ denotes the coefficient of $c$ in the \emph{slide product} $a \shuffle b$.
\end{theorem}

Here we give a more direct interpretation with a greatly simplified proof using well-labelled, flagged $(\P,\rho)$-partitions. We use Proposition~\ref{prop:well} to prove the disjoint union of the two linear orders is well-labelled and flagged, and then use Lemma~\ref{lem:flag_decomposition} to find the expression and recover Theorem~\ref{thm:slide}.

\begin{corollary}\label{cor:slide}
  Let $a, b$ be two weak compositions. Choose two linear posets $\mathcal{A}$ and $\mathcal{B}$ along with two well-labelled, flagged restriction maps $\alpha$ and $\beta$ such that $\fund_{(\mathcal{A},\alpha)} = \fund_a$ and $\fund_{(\mathcal{B},\beta)} = \fund_b$. Then
  \begin{equation}\label{e:slide2}
    \fund_{a} \fund_{b} = \fund_{(\mathcal{A},\alpha)} \fund_{(\mathcal{B},\beta)} = \fund_{(\mathcal{A}\oplus\mathcal{B},\alpha\oplus\beta)} =\sum_{\L \in \mathscr{L}(\mathcal{A}\oplus\mathcal{B})} \fund_{\red(\L,\alpha\oplus\beta)},
  \end{equation}
  where $\mathcal{A} \oplus \mathcal{B}$ denotes the partial order as the disjoint union of $\mathcal{A}$ and $\mathcal{B}$, i.e. no element of $\mathcal{A}$ is comparable to an element of $\mathcal{B}$. The set $\mathscr{L}(\mathcal{A} \oplus \mathcal{B})$ corresponds to the slide product of $\mathcal{A}$ and $\mathcal{B}$.
\end{corollary}

\begin{example}\label{ex:slide}
  Let $a = (0,0,2)$ and $b = (0,2,0)$, and take $(\P,\rho)$ as in Figure~\ref{fig:slide-prod}.
  This gives the following expansion for fundamental slide polynomials as
  \[ \fund_{(0,0,2)} \fund_{(0,2,0)} = \fund_{(0,2,2)} + \fund_{(1,2,1)} + \fund_{(0,3,1)} + \fund_{(1,3,0)} + \fund_{(2,2,0)} + \fund_{(0,4,0)} .\]
  Note that this generalizes Example~\ref{ex:shuffle} for the shuffle product $F_{(2)}F_{(2)}$.
\end{example}

\begin{figure}[ht]
  \begin{center}
    \begin{tikzpicture}[
        roundnode/.style={circle, draw=black, thick, minimum size=1ex},
        label/.style={%
          postaction={ decorate,transform shape,
            decoration={ markings, mark=at position .5 with \node #1;}}}]
      \node[roundnode] at (0,3) (A2) {$2$};
      \node[roundnode] at (0,1) (A1) {$1$};
      \node[roundnode] at (1,2) (A4) {$4$};
      \node[roundnode] at (1,0) (A3) {$3$};
      \node[right = 1pt of A1] {\Blue{$\scriptstyle\leq 3$}};
      \node[right = 1pt of A2] {\Blue{$\scriptstyle\leq 3$}};
      \node[right = 1pt of A3] {\Blue{$\scriptstyle\leq 2$}};
      \node[right = 1pt of A4] {\Blue{$\scriptstyle\leq 2$}};
      \draw (A1) -- (A2) ;
      \draw (A3) -- (A4) ;
      \node[roundnode] at (3,3) (L2) {$2$};
      \node[roundnode] at (3,2) (L1) {$1$};
      \node[roundnode] at (3,1) (L4) {$4$};
      \node[roundnode] at (3,0) (L3) {$3$};
      \node[right = 1pt of L1] {\Blue{$\scriptstyle\leq 3$}};
      \node[right = 1pt of L2] {\Blue{$\scriptstyle\leq 3$}};
      \node[right = 1pt of L3] {\Blue{$\scriptstyle\leq 2$}};
      \node[right = 1pt of L4] {\Blue{$\scriptstyle\leq 2$}};
      \draw (L2) -- (L1) ;
      \draw (L1) -- (L4) ;
      \draw (L4) -- (L3) ;
      \node[roundnode] at (4.5,3) (La2) {$2$};
      \node[roundnode] at (4.5,2) (La4) {$4$};
      \node[roundnode] at (4.5,1) (La1) {$1$};
      \node[roundnode] at (4.5,0) (La3) {$3$};
      \node[right = 1pt of La1] {\Blue{$\scriptstyle\leq \cancel{3}\Red{2}$}};
      \node[right = 1pt of La2] {\Blue{$\scriptstyle\leq 3$}};
      \node[right = 1pt of La3] {\Blue{$\scriptstyle\leq \cancel{2}\Red{1}$}};
      \node[right = 1pt of La4] {\Blue{$\scriptstyle\leq 2$}};
      \draw (La2) -- (La4) ;
      \draw (La4) -- (La1) ;
      \draw (La1) -- (La3) ;
      \node[roundnode] at (6,3) (Lb2) {$2$};
      \node[roundnode] at (6,2) (Lb4) {$4$};
      \node[roundnode] at (6,1) (Lb3) {$3$};
      \node[roundnode] at (6,0) (Lb1) {$1$};
      \node[right = 1pt of Lb1] {\Blue{$\scriptstyle\leq \cancel{3}\Red{2}$}};
      \node[right = 1pt of Lb2] {\Blue{$\scriptstyle\leq 3$}};
      \node[right = 1pt of Lb3] {\Blue{$\scriptstyle\leq 2$}};
      \node[right = 1pt of Lb4] {\Blue{$\scriptstyle\leq 2$}};
      \draw (Lb2) -- (Lb4) ;
      \draw (Lb4) -- (Lb3) ;
      \draw (Lb3) -- (Lb1) ;
      \node[roundnode] at (7.5,3) (Lc4) {$4$};
      \node[roundnode] at (7.5,2) (Lc2) {$2$};
      \node[roundnode] at (7.5,1) (Lc1) {$1$};
      \node[roundnode] at (7.5,0) (Lc3) {$3$};
      \node[right = 1pt of Lc1] {\Blue{$\scriptstyle\leq \cancel{3}\Red{2}$}};
      \node[right = 1pt of Lc2] {\Blue{$\scriptstyle\leq \cancel{3}\Red{2}$}};
      \node[right = 1pt of Lc3] {\Blue{$\scriptstyle\leq \cancel{2}\Red{1}$}};
      \node[right = 1pt of Lc4] {\Blue{$\scriptstyle\leq 2$}};
      \draw (Lc4) -- (Lc2) ;
      \draw (Lc2) -- (Lc1) ;
      \draw (Lc1) -- (Lc3) ;
      \node[roundnode] at (9,3) (Ld4) {$4$};
      \node[roundnode] at (9,2) (Ld2) {$2$};
      \node[roundnode] at (9,1) (Ld3) {$3$};
      \node[roundnode] at (9,0) (Ld1) {$1$};
      \node[right = 1pt of Ld1] {\Blue{$\scriptstyle\leq \cancel{3}\Red{1}$}};
      \node[right = 1pt of Ld2] {\Blue{$\scriptstyle\leq \cancel{3}\Red{2}$}};
      \node[right = 1pt of Ld3] {\Blue{$\scriptstyle\leq \cancel{2}\Red{1}$}};
      \node[right = 1pt of Ld4] {\Blue{$\scriptstyle\leq 2$}};
      \draw (Ld4) -- (Ld2) ;
      \draw (Ld2) -- (Ld3) ;
      \draw (Ld3) -- (Ld1) ;
      \node[roundnode] at (10.5,3) (Le4) {$4$};
      \node[roundnode] at (10.5,2) (Le3) {$3$};
      \node[roundnode] at (10.5,1) (Le2) {$2$};
      \node[roundnode] at (10.5,0) (Le1) {$1$};
      \node[right = 1pt of Le1] {\Blue{$\scriptstyle\leq \cancel{3}\Red{2}$}};
      \node[right = 1pt of Le2] {\Blue{$\scriptstyle\leq \cancel{3}\Red{2}$}};
      \node[right = 1pt of Le3] {\Blue{$\scriptstyle\leq 2$}};
      \node[right = 1pt of Le4] {\Blue{$\scriptstyle\leq 2$}};
      \draw (Le4) -- (Le3) ;
      \draw (Le3) -- (Le2) ;
      \draw (Le2) -- (Le1) ;
    \end{tikzpicture}
  \end{center}
  \caption{\label{fig:slide-prod} The well-labelled flagged $(\P,\rho)$ for computing the slide product $\fund_{(0,0,2)} \fund_{(0,2,0)}$. We give here the six linear extensions together with $\rhomax{\L}$ for each.}
\end{figure}
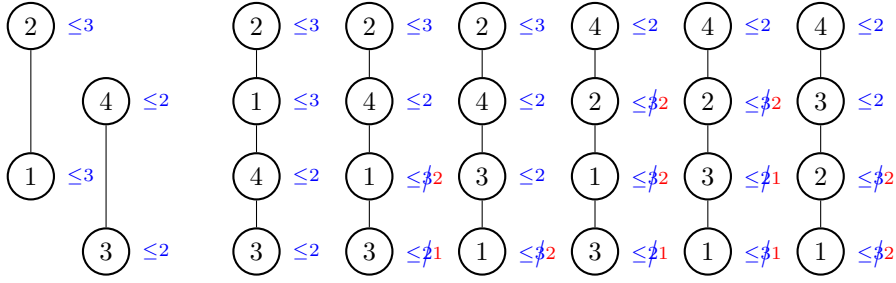

\subsubsection{Flagged Schur polynomials}
Also in \cite[Theorem~3.13]{AS17}, the authors show the positive expansion of Schubert polynomials in terms of slide polynomial. Lascoux and Sch\"{u}tzenberger~\cite{LS82} show that when a permutation is vexillary, the Schubert polynomial is a flagged Schur polynomials (see also~\cite{wachs82}). This shows that many flagged Schur polynomials are slide positive. We claim that all flagged Schur polynomials are slide positive, giving the restricted analogue to Corollary~\ref{cor:syt}. Let us recall what are flagged Schur polynomials.

Given a partition $\lambda$ and a \emph{flag} ${\bf b}=b_1\le b_2\le \ldots b_{\ell_{\lambda}}$, the \emph{flagged Schur function} $s_{\lambda,{\bf b}}$ is the sum of the monomials corresponding to semistandard tableau of shape $\lambda$ and entries in row $i$ bounded by $b_i$. As in the Schur case, we may express this in terms of $(\P,\rho)$-partitions as
\[ \fund_{(\P_\lambda,\rho_{\bf b})} = s_{\lambda,{\bf b}} , \]
where $\P_\lambda$ is the poset constructed as in Fig.~\ref{fig:partition} and $\rho_{\bf b}$ is the restriction given by $\rho_{\bf b}(x)=b_i$ if $x\in \P$ is in row $i$.
  
For example, $s_{431,(2,4,4)}$ is given by the $(\P,\rho)$-partitions with $\P=\P_\lambda$ and $\rho=\rho_{\bf b}$ as in Figure~\ref{fig:flagged_partition}. This is the flagged Schur polynomial equal to the Schubert polynomial indexed by the vexillary permutation $w=1637245$ (see~\cite{wachs82}).

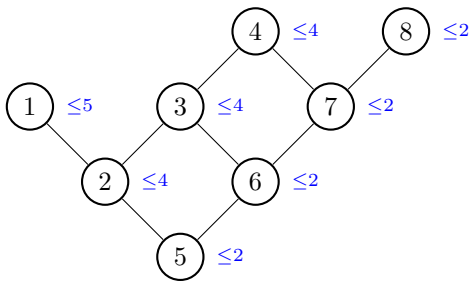
\begin{figure}[ht]
  \begin{center}
    \begin{tikzpicture}[
        roundnode/.style={circle, draw=black, thick, minimum size=1ex},
        label/.style={%
          postaction={ decorate,transform shape,
            decoration={ markings, mark=at position .5 with \node #1;}}}]
      \node[roundnode] at (3,3) (A5) {$4$};
      \node[roundnode] at (5,3) (A1) {$8$};
      \node[roundnode] at (0,2) (A8) {$1$};
      \node[roundnode] at (2,2) (A6) {$3$};
      \node[roundnode] at (4,2) (A2) {$7$};
      \node[roundnode] at (1,1) (A7) {$2$};
      \node[roundnode] at (3,1) (A3) {$6$};
      \node[roundnode] at (2,0) (A4) {$5$};
      \node[right = 1pt of A8] {\Blue{$\scriptstyle\leq 5$}};
      \node[right = 1pt of A7] {\Blue{$\scriptstyle\leq 4$}};
      \node[right = 1pt of A6] {\Blue{$\scriptstyle\leq 4$}};
      \node[right = 1pt of A5] {\Blue{$\scriptstyle\leq 4$}};
      \node[right = 1pt of A4] {\Blue{$\scriptstyle\leq 2$}};
      \node[right = 1pt of A3] {\Blue{$\scriptstyle\leq 2$}};
      \node[right = 1pt of A2] {\Blue{$\scriptstyle\leq 2$}};
      \node[right = 1pt of A1] {\Blue{$\scriptstyle\leq 2$}};
      \draw (A4) -- (A3) ;
      \draw (A3) -- (A2) ;
      \draw (A2) -- (A1) ;
      \draw (A7) -- (A6) ;
      \draw (A6) -- (A5) ;
      \draw (A4) -- (A7) ;
      \draw (A7) -- (A8) ;
      \draw (A3) -- (A6) ;
      \draw (A2) -- (A5) ;
    \end{tikzpicture}
  \end{center}
  \caption{\label{fig:flagged_partition}The labeled poset corresponding to the partition $(4,3,1)$ with $\rho$ restriction given by the flag $2,4,4$.}
\end{figure}

\begin{theorem}\label{thm:flagschur_in_slide}
  For $\lambda$ a partition and ${\bf b}=b_1\le  b_2\le \cdots \le b_{\ell(\lambda)}$ a flag of integers, we have
  \begin{equation}\label{e:flaggedSchur_slide}
    s_{\lambda,{\bf b}} = \sum_{T \in \mathrm{SYT}(\lambda)} \fund_{(\L_T,\rho_{\bf b})} = \sum_{T \in \mathrm{SYT}(\lambda)} \fund_{\red((\L_T,\rho_{\bf b}))} ,
  \end{equation}
  where a standard tableau $T$ determines a unique linear extension $\L_T$ of $\P_\lambda$. 
\end{theorem}

\begin{proof}
  It is clear that the pair $(\P_\lambda,\rho_{\bf b})$ is well-labelled and flagged, and so the formula follows from Corollary~\ref{cor:fund}. 
\end{proof}

Any term $\fund_{\L_T,\rhomax{\L_T}} $ is zero if the first entry of $,\rhomax{\L_T}$ is $<1$ and can be removed from the sum. The remaining terms are all flagged, hence the summation is a sum of slide polynomials. 

\begin{remark}
  Reiner and Shimozono \cite[Theorem~23]{RS95} show that any flagged Schur polynomial is a single key polynomial. More generally, they show that any flagged skew Schur polynomial is a positive sum of key polynomials. More recently,  Assaf and Searles~\cite[Theorem~2.13]{AS18} show that any key polynomial is a positive sum of slide polynomials. Combining those two results yields any flagged skew Schur function expands positively in term of slide polynomials. We offer here a more immediate approach to this result by using flagged $(\P,\rho)$-partitions.
\end{remark}

\begin{example}
  Consider $\lambda=(3,2)$ and ${\bf b}=(2,6)$ (see Figure~\ref{fig:flaggedSchur}). Each standard tableau $T$ gives rise to a linear extension $\L_T$ of $\P_{\lambda}$.
  We then compute $\rhomax{\L_T}$ (in red in the figure). If the first component of $\rhomax{\L_T}$ is $<1$, then we remove this linear order.
  In the Figure, the last linear order has component $0$. Hence it will not contribute to the expansion. The four other linear orders are flagged $(\L_T,\rhomax{\L_T})$-partitions and thus  slide polynomials. We then have
  $$ s_{(3,2),(2,6)} =  \fund_{(0,3,0,0,0,2)} + \fund_{(2,2,0,0,0,1)} + \fund_{(1,3,0,0,0,1)} + \fund_{(2,3)}$$
  as they appear from left to right in Figure~\ref{fig:flaggedSchur}.
\end{example}

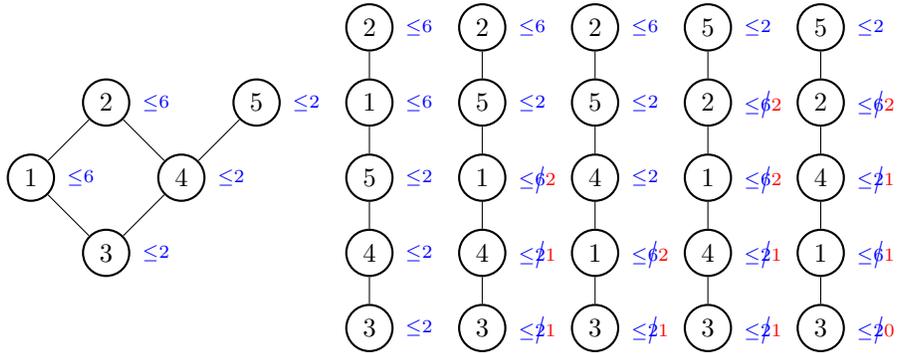
\begin{figure}[ht]
  \begin{center}
    \begin{tikzpicture}[
        roundnode/.style={circle, draw=black, thick, minimum size=1ex},
        label/.style={%
          postaction={ decorate,transform shape,
            decoration={ markings, mark=at position .5 with \node #1;}}}]
      \node[roundnode] at (2,2) (A2) {$2$};
      \node[roundnode] at (4,2) (A5) {$5$};
      \node[roundnode] at (1,1) (A1) {$1$};
      \node[roundnode] at (3,1) (A4) {$4$};
      \node[roundnode] at (2,0) (A3) {$3$};
      \node[right = 1pt of A1] {\Blue{$\scriptstyle\leq 6$}};
      \node[right = 1pt of A2] {\Blue{$\scriptstyle\leq 6$}};
      \node[right = 1pt of A3] {\Blue{$\scriptstyle\leq 2$}};
      \node[right = 1pt of A4] {\Blue{$\scriptstyle\leq 2$}};
      \node[right = 1pt of A5] {\Blue{$\scriptstyle\leq 2$}};
      \draw (A1) -- (A2) ;
      \draw (A3) -- (A4) ;
      \draw (A4) -- (A5) ;
      \draw (A3) -- (A1) ;
      \draw (A4) -- (A2) ;
      \node[roundnode] at (5.5,3) (L2) {$2$};
      \node[roundnode] at (5.5,1) (L5) {$5$};
      \node[roundnode] at (5.5,2) (L1) {$1$};
      \node[roundnode] at (5.5,0) (L4) {$4$};
      \node[roundnode] at (5.5,-1) (L3) {$3$};
      \node[right = 1pt of L1] {\Blue{$\scriptstyle\leq 6$}};
      \node[right = 1pt of L2] {\Blue{$\scriptstyle\leq 6$}};
      \node[right = 1pt of L3] {\Blue{$\scriptstyle\leq 2$}};
      \node[right = 1pt of L4] {\Blue{$\scriptstyle\leq 2$}};
      \node[right = 1pt of L5] {\Blue{$\scriptstyle\leq 2$}};
      \draw (L3) -- (L4) ;
      \draw (L4) -- (L5) ;
      \draw (L5) -- (L1) ;
      \draw (L1) -- (L2) ;
      \node[roundnode] at (7,3) (La2) {$2$};
      \node[roundnode] at (7,2) (La5) {$5$};
      \node[roundnode] at (7,1) (La1) {$1$};
      \node[roundnode] at (7,0) (La4) {$4$};
      \node[roundnode] at (7,-1) (La3) {$3$};
      \node[right = 1pt of La1] {\Blue{$\scriptstyle\leq \cancel{6}\Red{2}$}};
      \node[right = 1pt of La2] {\Blue{$\scriptstyle\leq 6$}};
      \node[right = 1pt of La3] {\Blue{$\scriptstyle\leq \cancel{2}\Red{1}$}};
      \node[right = 1pt of La4] {\Blue{$\scriptstyle\leq \cancel{2}\Red{1}$}};
      \node[right = 1pt of La5] {\Blue{$\scriptstyle\leq 2$}};
      \draw (La3) -- (La4) ;
      \draw (La4) -- (La1) ;
      \draw (La1) -- (La5) ;
      \draw (La5) -- (La2) ;
     \node[roundnode] at (8.5,3) (Lb2) {$2$};
      \node[roundnode] at (8.5,2) (Lb5) {$5$};
      \node[roundnode] at (8.5,0) (Lb1) {$1$};
      \node[roundnode] at (8.5,1) (Lb4) {$4$};
      \node[roundnode] at (8.5,-1) (Lb3) {$3$};
      \node[right = 1pt of Lb1] {\Blue{$\scriptstyle\leq \cancel{6}\Red{2}$}};
      \node[right = 1pt of Lb2] {\Blue{$\scriptstyle\leq 6$}};
      \node[right = 1pt of Lb3] {\Blue{$\scriptstyle\leq \cancel{2}\Red{1}$}};
      \node[right = 1pt of Lb4] {\Blue{$\scriptstyle\leq 2$}};
      \node[right = 1pt of Lb5] {\Blue{$\scriptstyle\leq 2$}};
      \draw (Lb3) -- (Lb1) ;
      \draw (Lb1) -- (Lb4) ;
      \draw (Lb4) -- (Lb5) ;
      \draw (Lb5) -- (Lb2) ;
     \node[roundnode] at (10,2) (Lc2) {$2$};
      \node[roundnode] at (10,3) (Lc5) {$5$};
      \node[roundnode] at (10,1) (Lc1) {$1$};
      \node[roundnode] at (10,0) (Lc4) {$4$};
      \node[roundnode] at (10,-1) (Lc3) {$3$};
      \node[right = 1pt of Lc1] {\Blue{$\scriptstyle\leq \cancel{6}\Red{2}$}};
      \node[right = 1pt of Lc2] {\Blue{$\scriptstyle\leq \cancel{6}\Red{2}$}};
      \node[right = 1pt of Lc3] {\Blue{$\scriptstyle\leq \cancel{2}\Red{1}$}};
      \node[right = 1pt of Lc4] {\Blue{$\scriptstyle\leq \cancel{2}\Red{1}$}};
      \node[right = 1pt of Lc5] {\Blue{$\scriptstyle\leq 2$}};
      \draw (Lc3) -- (Lc4) ;
      \draw (Lc4) -- (Lc1) ;
      \draw (Lc1) -- (Lc2) ;
      \draw (Lc2) -- (Lc5) ;
     \node[roundnode] at (11.5,2) (Ld2) {$2$};
      \node[roundnode] at (11.5,3) (Ld5) {$5$};
      \node[roundnode] at (11.5,0) (Ld1) {$1$};
      \node[roundnode] at (11.5,1) (Ld4) {$4$};
      \node[roundnode] at (11.5,-1) (Ld3) {$3$};
      \node[right = 1pt of Ld1] {\Blue{$\scriptstyle\leq \cancel{6}\Red{1}$}};
      \node[right = 1pt of Ld2] {\Blue{$\scriptstyle\leq \cancel{6}\Red{2}$}};
      \node[right = 1pt of Ld3] {\Blue{$\scriptstyle\leq \cancel{2}\Red{0}$}};
      \node[right = 1pt of Ld4] {\Blue{$\scriptstyle\leq \cancel{2}\Red{1}$}};
      \node[right = 1pt of Ld5] {\Blue{$\scriptstyle\leq 2$}};
      \draw (Ld3) -- (Ld1) ;
      \draw (Ld1) -- (Ld4) ;
      \draw (Ld4) -- (Ld2) ;
      \draw (Ld2) -- (Ld5) ;
    \end{tikzpicture}
  \end{center}
  \caption{\label{fig:flaggedSchur} The well-labelled flagged $(\P,\rho)$ partitions for  $\lambda=(3,2)$ and $\rho$ restriction given by the flag $1,6$. We give here the five linear extension together with $\rhomax{\L_T}$ for each standard tableau.}
\end{figure}

%
%

\bibliographystyle{amsalpha} 
\bibliography{P-partitions}

\end{document}